\documentclass{amsart}

\usepackage[latin1]{inputenc}
\usepackage[english]{babel}
\usepackage{array}

\usepackage[T1]{fontenc}
\usepackage{mathrsfs}
\usepackage{graphics,psfrag}
\usepackage{graphicx}
\usepackage{latexsym}
\usepackage{amssymb}
\usepackage{amsmath}
\usepackage[mathscr]{eucal}
\usepackage[all]{xy}
\usepackage{epsfig}
\usepackage{enumerate}
\usepackage{stmaryrd}
\usepackage{url}

\RequirePackage{color}
\definecolor{myred}{rgb}{0.75,0,0}
\definecolor{mygreen}{rgb}{0,0.5,0}
\definecolor{myblue}{rgb}{0,0,0.65}

\RequirePackage{ifpdf}
\ifpdf
  \IfFileExists{pdfsync.sty}{\RequirePackage{pdfsync}}{}
  \RequirePackage[pdftex,
   colorlinks = true,
   urlcolor = myred, 
   citecolor = mygreen, 
   linkcolor = myred, 
   pdftitle = {Manuscript},
   pdfauthor = {Daniel Juteau},
   pagebackref,
   pdfpagemode=None,
   bookmarksopen=true]{hyperref}
\else
  \RequirePackage[hypertex]{hyperref}
\fi

\RequirePackage{ae, aecompl, aeguill} 

\newcommand\T{\rule{0pt}{3.1ex}}

\newcommand\B{\rule[-1.7ex]{0pt}{0pt}}

\def\wh{\widehat}
\def\wt{\widetilde}
\def\ov{\overline}
\def\un{\underline}

    \def\AM{{\mathbb{A}}}
    
\def\CG{{\mathfrak C}}    \def\CM{{\mathbb{C}}}
\def\DG{{\mathfrak D}}    
    \def\EM{{\mathbb{E}}}
    \def\FM{{\mathbb{F}}}
  \def\gG{{\mathfrak g}}  \def\GM{{\mathbb{G}}}
    
\def\IG{{\mathfrak I}}    
    
    \def\KM{{\mathbb{K}}}
\def\LG{{\mathfrak L}}    \def\LM{{\mathbb{L}}}
  \def\mG{{\mathfrak m}}  
    \def\NM{{\mathbb{N}}}
\def\OG{{\mathfrak O}}    \def\OM{{\mathbb{O}}}
\def\PG{{\mathfrak P}}    \def\PM{{\mathbb{P}}}
    \def\QM{{\mathbb{Q}}}
    
\def\SG{{\mathfrak S}}    
\def\TG{{\mathfrak T}}

\def\XG{{\mathfrak X}}    
\def\YG{{\mathfrak Y}}    
    \def\ZM{{\mathbb{Z}}}

    \def\AC{{\mathcal{A}}}
    \def\BC{{\mathcal{B}}}
    \def\CC{{\mathcal{C}}}
    \def\DC{{\mathcal{D}}}
    
    \def\FC{{\mathcal{F}}}
    
    \def\HC{{\mathcal{H}}}
    
    \def\JC{{\mathcal{J}}}
    
    \def\LC{{\mathcal{L}}}
    \def\MC{{\mathcal{M}}}
    \def\NC{{\mathcal{N}}}
    \def\OC{{\mathcal{O}}}
    
    \def\QC{{\mathcal{Q}}}
    
    \def\SC{{\mathcal{S}}}
    \def\TC{{\mathcal{T}}}
    \def\UC{{\mathcal{U}}}
    \def\VC{{\mathcal{V}}}

\def\Xti{{\tilde{X}}}

\def\a{\alpha}
\def\b{\beta}

\def\G{\Gamma}

\def\D{\Delta}
\def\e{\varepsilon}

\def\l{\lambda}

\def\t{\tau}

\DeclareMathOperator{\Aut}{{\mathrm{Aut}}}

\DeclareMathOperator{\codim}{codim}
\DeclareMathOperator{\Coker}{{\mathrm{Coker}}}

\DeclareMathOperator{\Cone}{Cone}

\DeclareMathOperator{\Gal}{{\mathrm{Gal}}}
\DeclareMathOperator{\Hom}{{\mathrm{Hom}}}

\DeclareMathOperator{\ic}{{\mathcal{IC}}}
\DeclareMathOperator{\Id}{{\mathrm{Id}}}
\DeclareMathOperator{\id}{{\mathrm{Id}}}
\DeclareMathOperator{\im}{{\mathrm{Im}}}

\DeclareMathOperator{\Irr}{{\mathrm{Irr}}}
\DeclareMathOperator{\Ker}{{\mathrm{Ker}}}
\DeclareMathOperator{\Loc}{{\mathrm{Loc}}}

\DeclareMathOperator{\RHOM}{R\underline{Hom}}
\DeclareMathOperator{\rg}{R\G}

\DeclareMathOperator{\Sh}{{\mathrm{Sh}}}
\DeclareMathOperator{\Soc}{Soc}
\DeclareMathOperator{\Spec}{{\mathrm{Spec}}}

\DeclareMathOperator{\tr}{{\mathrm{tr}}}
\DeclareMathOperator{\Top}{Top}
\DeclareMathOperator{\Tor}{{\mathrm{Tor}}}

\newcommand{\elem}[1]{\stackrel{#1}{\longto}}
\newcommand{\map}[1]{\stackrel{#1}{\to}}
\newcommand{\leftelem}[1]{\stackrel{#1}{\longfrom}}

\def\Imp{\Longrightarrow}

\def\Iff{\Longleftrightarrow}
\def\to{\rightarrow}

\def\longto{\longrightarrow}
\def\longfrom{\longleftarrow}
\def\injto{\hookrightarrow}
\def\onto{\twoheadrightarrow}
\def\rtordu{\rightsquigarrow}
\def\isom{\stackrel{\sim}{\to}}

\def\SS{\scriptstyle}

\def\lexp#1#2{\kern\scriptspace\vphantom{#2}^{#1}\kern-\scriptspace#2}

\mathchardef\inferieur="321E
\mathchardef\superieur="321F

\def\eqna{\begin{eqnarray*}}
\def\endeqna{\end{eqnarray*}}

\def\ie{{\textit{i}.\textit{e}.}}

\def\op{\text{op}}

\def\p{{}^p\!}
\def\pp{{}^{p_+}\!}
\def\0{{}^0\!}

\def\F{\FM\,}

\def\di{\mathrm{div}}
\def\tors{{\mathrm{tors}}}
\def\free{{\mathrm{free}}}
\def\mini{{\mathrm{min}}}
\def\jmin{{j_\mini}}
\def\io{{i_0}}

\def\reg{{\mathrm{reg}}}
\def\jreg{{j_\reg}}
\def\isubreg{{i_\subreg}}

\def\subreg{{\mathrm{subreg}}}

\def\Sing{{\mathrm{Sing}}}

\def\octa#1#2#3#4#5#6#7{
\xymatrix #7{
&&&\\
\\
&& #3 \ar[ddr] \ar[uur]\\
&&&&&\\
& #2 \ar[uur] \ar[dr] && #5 \ar[ddr] \ar[ur]\\
&& #4 \ar[drr] \ar[ur]\\
#1 \ar[uur] \ar[urr] &&&& #6 \ar[ddr] \ar[drr]\\
&&&&&&\\
&&&&&
}
}

\usepackage{lmodern}

\theoremstyle{plain}
\newtheorem{theorem}{Theorem}[section]
\newtheorem{proposition}[theorem]{Proposition}
\newtheorem{lemma}[theorem]{Lemma}
\newtheorem{corollary}[theorem]{Corollary}

\theoremstyle{definition}
\newtheorem{definition}[theorem]{Definition}
\newtheorem{assumption}[theorem]{Assumption}

\theoremstyle{remark}
\newtheorem{remark}[theorem]{Remark}
\newtheorem{example}[theorem]{Example}

\title{Decomposition numbers for perverse sheaves}

\author{Daniel Juteau}

\address{Mathematical Sciences Research Institute\\ 
17 Gauss Way\\
Berkeley, CA 94720 (USA)}

\email{juteau@math.jussieu.fr}

\begin{document}

\begin{abstract}
The purpose of this article is to set foundations for decomposition
numbers of perverse sheaves, to give some methods to calculate them in
simple cases, and to compute them concretely in two situations:
for a simple (Kleinian) surface singularity, and for the closure of the
minimal non-trivial orbit in a simple Lie algebra.

This work has applications to modular representation theory, for Weyl
groups using the nilpotent cone of the corresponding semisimple Lie
algebra, and for reductive algebraic group schemes using the affine
Grassmannian of the Langlands dual group.
\end{abstract}

\maketitle

\section{Introduction}

The purpose of this article is to set foundations for decomposition
numbers of perverse sheaves, to give some methods to calculate them in
simple cases, and to compute them concretely for simple and minimal
singularities.

We consider varieties over $\ov\FM_p$, and perverse sheaves with
coefficients in $\EM$, where $\EM$ is one of the rings in an
$\ell$-modular system $(\KM,\OM,\FM)$, where $\ell$ is a prime
different from $\ell$. These notions are explained in Subsection
\ref{subsec:context}.

Modular systems were introduced in modular
representation theory of finite groups. The idea is that we use a ring
of integers $\OM$ to go from a field $\KM$ of characteristic zero to a
field $\FM$ of characteristic $\ell$. For a finite group $W$, we define
the decomposition numbers $d^W_{EF}$, for $E \in \Irr\KM W$ and
$F \in \Irr\FM W$, by $d^W_{EF} = [\FM \otimes_\OM E_\OM : F]$,
where $\EM_\OM$ is a $W$-stable $\OM$-lattice in $E$ (this
multiplicity is well-defined). In many cases (for example, for the
symmetric group), the ordinary irreducibles (over $\KM$) are known,
but the modular ones (over $\FM$) are not. Then the problem of
determining the modular characters is equivalent to the problem of
determining the decomposition matrix $D^W=(d^W_{EF})$.

We can do the same for perverse sheaves on some variety $X$:
we can define decomposition
numbers $d^X_{(\OC,\LC),(\OC',\LC')}$, where $(\OC,\LC)$ and
$(\OC',\LC')$ are pairs consisting of smooth irreducible locally
closed subvariety and an irreducible $\EM$ local system on it, 
for $\EM = \KM$ or $\FM$ respectively. The simple perverse sheaves are
indexed by such pairs (if we fix a stratification, we take strata for
$\OC$ and $\OC'$). They are intersection cohomology complexes. As in
modular representation theory, one can take an integral form and
apply the functor of modular reduction $\FM \otimes_\OM^L -$.

In \cite{these}, it has been shown that the decomposition matrix of a
Weyl group can be extracted from a decomposition matrix for
equivariant perverse sheaves on the nilpotent cone. This required
to define a modular Springer correspondence, using a Fourier-Deligne
transform (I will explain this in a forthcoming article).

Thus it is very desirable to be able to calculate decomposition
numbers for equivariant perverse sheaves on the nilpotent cone. The singularity of
the nilpotent cone along the subregular orbit is a simple surface singularity
\cite{BRI,SLO1,SLO2}. At the other extreme, one can look at the
singularity of the closure of a minimal non-trivial nilpotent orbit at
the origin. These two cases are treated here.

On the other hand,
by the results of \cite{MV}, the decomposition numbers for a reductive
algebraic group scheme can be interpreted as decomposition numbers for
equivariant perverse sheaves on the affine Grassmannian of the Langlands dual
group. Moreover, most of the minimal degenerations of this 
(infinite-dimensional) variety are
simple or minimal singularities \cite{MOV}, so the calculations that
we carry out in this article can be used to recover some decomposition
numbers for reductive algebraic group schemes geometrically. This will be done in
another article, where we will also explain that one can go in the
other direction and prove geometric results using known decomposition
numbers.

In the author's opinion, perverse sheaves over rings of integers and
in positive characteristic, and their decomposition numbers, will
prove to be useful in many ways. For simple and minimal singularities,
we already have two different applications to modular representation theory.
So it seemed desirable to show how to calculate these decomposition
numbers independently of the framework of Springer correspondence.

Now let us give an outline of the article. Section \ref{sec:perverse}
contains the technical preliminaries. First, we set the context and
recall the definition of perverse sheaves over $\KM$, $\OM$, $\FM$. 
The treatment of $\OM$-coefficients in the standard reference \cite{BBD}
is done in two pages (\S\ 3.3). Over a field, the middle perversity
$p$ is self-dual, but here one has to consider two perversities, $p$
and $p_+$, exchanged by the duality. The cause of the trouble is
torsion. It seemed worthwhile to explain this construction in a more
general context.
Given an abelian category with a torsion theory,
there is a known procedure to construct another abelian category lying
inside the derived category \cite{HRS}. Our point of view is slightly
different: we start we a $t$-category, and we assume that its heart is
endowed with a torsion theory. Then we can construct a new
$t$-structure on the same triangulated category.
After recalling the notion of $t$-structure, we study the interaction
between torsion theories and $t$-structures. Then, we recall the
notion of recollement and its properties (most can be found in
\cite{BBD}), and we see how it interacts with torsion theories.
Then we see why the $t$-structure defining perverse sheaves is indeed
a $t$-structure, thus justifying the definition we recalled before.
In this context, we have functors of extension of scalars
$\KM\otimes_\OM -$ and of modular reduction $\FM\otimes^L_\OM -$.
One of the main technical points is that truncations do not commute
with modular reduction. We study carefully the failure of
commutativity of these functors, because this is precisely what will give
rise to non-trivial decomposition numbers, in the setting of
recollement. Then it is time to define these decomposition numbers for
perverse sheaves, and finally we deal with equivariance.

Since we can translate some problems of modular representation
theory in terms of decomposition numbers for perverse sheaves, it is
very important to be able to compute them. In general, it should be
very difficult. In Section \ref{sec:techniques}, we give some
techniques to compute them in certain cases. It is enough to determine
the intersection cohomology stalks over $\FM$ (in the applications,
they are usually known over $\KM$). In characteristic zero, a lot of
information can be obtained from the study of semi-small and small
proper separable morphisms. We explain what is still true in
characteristic $\ell$, but also why it is less useful, unless we have
a small resolution of singularities. Then we recall the notion of
$\EM$-smoothness, and we give some conditions which imply that some
decomposition numbers are zero. This is the simplest case, where the
intersection cohomology complex is just the constant sheaf (suitably
shifted). In general, we do not have many tools at our disposal, so
Deligne's construction, which works in any case, is very important to
do calculations in the modular setting. When we have an isolated cone
singularity, or more generally an isolated singularity on an affine
variety endowed with a $\GM_m$-action contracting to the origin, it is
much more likely to be handled. Finally, we recall the notion of
smooth equivalence of singularities. We can use the
results about a singularity to study a smoothly equivalent one. When
we deal only with constant local systems, this even gives all the
information. In general, one has to get extra information to determine
all the decomposition numbers.

In Section \ref{sec:simple}, we determine the decomposition numbers for
simple (or Kleinian) surface singularities. Their geometry has been studied a
lot. It is a nice illustration of the theory and techniques described
earlier to do this calculation, using geometrical results in the
literature. By a famous theorem of Brieskorn and Slodowy
\cite{BRI, SLO1, SLO2}, the singularity of the nilpotent cone of a
simple Lie algebra along the subregular orbit is a simple
singularity. This is an instance where we can determine all the
decomposition numbers, even for non constant local systems, using a
smooth equivalence of singularities, thanks to Slodowy's study of the
symmetries of the minimal resolutions of simple singularities (thus
giving a meaning to simple singularities of non-homogeneous type).

Finally, in Section \ref{sec:min}, we determine the decomposition
numbers for closures of minimal non-trivial orbits in simple Lie
algebras. Again, this is a nice illustration of the previous parts (it
is an isolated cone singularity). This result uses the determination
of the integral cohomology of the minimal orbit, which we obtained in
a previous article \cite{cohmin}.


\setcounter{tocdepth}{2}
\tableofcontents

\section{Perverse sheaves over $\KM$, $\OM$ and $\FM$}
\label{sec:perverse}

\subsection{Context}
\label{subsec:context}

In all this article, we fix on the one hand a prime number $p$
and an algebraic closure $\ov\FM_p$ of the prime field with $p$ elements,
and for each power $q$ of $p$, we denote by $\FM_q$ the unique subfield
of $\ov\FM_p$ with $q$ elements. On the other hand, we fix a prime number $\ell$
distinct from $p$, and a finite extension $\KM$ of the field $\QM_\ell$
of $\ell$-adic numbers, whose valuation ring we denote by $\OM$.
Let $\mG = (\varpi)$ be the maximal ideal of $\OM$, and let $\FM = \OM/\mG$
be its residue field (which is finite of characteristic $\ell$).
In modular representation theory, a triple such as $(\KM,\OM,\FM)$
is called an $\ell$-modular system. The letter $\EM$ will often be used
to denote either of these three rings.

Let $k$ denote $\FM_q$ or $\ov\FM_p$
(we could have taken the field $\CM$ of complex numbers instead, and
then we could have used arbitrary coefficients; however, for future
applications, we will need to treat the positive characteristic case,
with the étale topology). We will consider only
separated $k$-schemes of finite type, and morphisms of
$k$-schemes. Such schemes will be called varieties. If $X$ is a
variety, we will say ``$\EM$-sheaves on $X$'' for ``constructible
$\EM$-sheaves on $X$''.  We will denote by $\Sh(X,\EM)$ the noetherian
abelian category of $\EM$-sheaves on $X$, and by $\Loc(X,\EM)$ the
full subcategory of $\EM$-local systems on $X$. If $X$ is connected,
these correspond to the continuous representations of the étale
fundamental group of $X$ at any base point.

Let $D^b_c(X,\EM)$ be the bounded derived category of $\EM$-sheaves as
defined by Deligne.  The category $D^b_c(X,\EM)$ is triangulated, and
endowed with a $t$-structure whose heart is equivalent to the abelian
category of $\EM$-sheaves, because the following condition is
satisfied \cite{BBD,DELIGNE}.
\begin{equation}
\begin{array}{l}
\text{For each finite extension } k' \text{ of } k \text{ contained in } \ov\FM_p,\\
\text{the groups $H^i(\Gal(\ov\FM_p/k'),\ZM/\ell)$, $i\in\NM$, are finite.}
\end{array}
\end{equation}
We call this $t$-structure the \emph{natural} $t$-structure on
$D^b_c(X,\EM)$. The notion of $t$-structure will be recalled in the
next section. For triangulated categories and derived categories, we
refer to \cite{KS2,WEIBEL}.

We have internal operations $\otimes^\LM_\EM$ and $\RHOM$ on $D^b_c(X,\EM)$,
and, if $Y$ is another scheme, for $f : X \to Y$ a morphism we have triangulated functors
\begin{gather*}
f_!,\, f_* : D^b_c(X,\EM) \to D^b_c(Y,\EM)\\
f^*,\, f^! : D^b_c(Y,\EM) \to D^b_c(X,\EM)
\end{gather*}
We omit the letter $R$ which is normally used (\emph{e.g.} $Rf_*$,
$Rf_!$) meaning that we consider derived functors. For the functors
between categories of sheaves, we will use a $0$ superscript, as in
$\0 f_*$ and $\0 f_!$, following \cite{BBD}.

We will denote by 
\[
\DC_{X,\EM} : D^b_c(X,\EM)^\op \to D^b_c(X,\EM)
\]
the dualizing functor $\DC_{X,\EM} (-) = \RHOM(-,a^!\EM)$, where $a:X\to\Spec k$
is the structural morphism.

We have a modular reduction functor
$\FM \otimes^\LM_\OM (-) : D^b_c(X,\OM) \to D^b_c(X,\FM)$, which we will
simply denote by $\FM(-)$.
It is triangulated, and it commutes with the functors 
$f_!$, $f_*$, $f^*$, $f^!$ and the duality.
Moreover, it maps a torsion-free sheaf to a sheaf, and a torsion sheaf
to a complex concentrated in degrees $-1$ and $0$.

By definition, we have $D^b_c(X,\KM) = \KM \otimes_\OM D^b_c(X,\OM)$,
and $\Sh(X,\KM) = \KM \otimes_\OM \Sh(X,\OM)$. The functor
$\KM \otimes_\OM (-) : D^b_c(X,\OM) \to D^b_c(X,\KM)$ is exact.

In this section, we are going to recall the construction of the
perverse $t$-structure on $D^b_c(X,\EM)$ for the middle perversity $p$
(with two versions over $\OM$, where we have two perversities $p$ and
$p_+$ exchanged by the duality). We will recall the main points of the
treatment of $t$-structures and recollement of \cite{BBD}, to which we
refer for the details. However, in this work we emphasize the aspects
concerning $\OM$-sheaves, and we give some complements.

Before going through all these general constructions, let us already see
what these perverse sheaves are. They form an abelian full subcategory
$\p \MC(X,\EM)$ of $D^b_c(X,\EM)$. If $\EM$ is $\KM$ or $\FM$,
then this abelian category is artinian and noetherian, and its simple
objects are of the form $\p j_{!*}(\LC[\dim V)])$, where $j : V \to X$ is
the inclusion of a smooth irreducible subvariety, $\LC$ is an
irreducible locally constant constructible $\EM$-sheaf on $V$, and
$\p j_{!*}$ the intermediate extension functor. If $\EM = \OM$, this
abelian category is only noetherian.
In any case, $\p \MC(X,\EM)$ is the intersection of the full
subcategories
$\p D^{\leqslant 0}(X,\EM)$ and $\p D^{\geqslant 0}(X,\EM)$
of $D^b_c(X,\EM)$, where, if $A$ is a complex in $\DC^b_c(X,\EM)$, we have
\begin{multline}
\label{def:p <= 0}
A\in \p D^{\leqslant 0}(X,\EM)  \Iff
\text{for all points $x$ in $X$, }\\
\HC^n i_x^* A = 0
\text{ for all } n > - \dim(x)
\end{multline}
\begin{multline}
\label{def:p >= 0}
A\in \p D^{\geqslant 0}(X,\EM) \Iff
\text{for all points $x$ in $X$, }\\
\HC^n i_x^! A = 0
\text{ for all } n < - \dim(x)
\end{multline}
Here the points are not necessarily closed,
$i_x$ is the inclusion of $x$ into $X$, and
$\dim(x) = \dim \ov{\{x\}} = \deg\tr(k(x)/k)$.

The pair $(\p D^{\leqslant 0}, \p D^{\geqslant 0})$ is a $t$-structure
on $D^b_c(X,\EM)$, and $\p \MC(X,\EM)$ is its \emph{heart}.

When $\EM$ is a field (\ie\ $\EM = \KM$ or $\FM$),
the duality functor $\DC_{X,\EM}$ exchanges
$\p D^{\leqslant 0}(X,\EM)$ and $\p D^{\geqslant 0}(X,\EM)$,
so it induces a self-duality on $\p\MC(X,\EM)$.

However, when $\EM = \OM$, this is no longer true.
The perversity $p$ is not self-dual in this case. The duality
exchanges the $t$-structure defined by the middle perversity $p$ with
the $t$-structure
$(\pp D^{\leqslant 0}(X,\OM),\pp D^{\geqslant 0}(X,\OM))$
defined by
\begin{multline}
\label{def:p+ <= 0}
A\in \pp D^{\leqslant 0}(X,\OM) \Iff
\text{for all points $x$ in $X$, }\\
\begin{cases}
\HC^n i_x^* A = 0 \text{ for all } n > - \dim(x) + 1\\
\HC^{- \dim(x) + 1} i_x^* A \text{ is torsion}
\end{cases}
\end{multline}
\begin{multline}
\label{def:p+ >= 0}
A\in \pp D^{\geqslant 0}(X,\OM) \Iff
\text{for all points $x$ in $X$, }\\
\begin{cases}
\HC^n i_x^! A = 0 \text{ for all } n < - \dim(x)\\
\HC^{- \dim(x)} i_x^! A \text{ is torsion-free}
\end{cases}
\end{multline}
The definition of torsion (resp. torsion-free) objects is given in
Definition \ref{def:torsion}. 
We say that this $t$-structure is defined by the perversity $p_+$, and
that the duality exchanges $p$ and $p_+$. We denote by $\pp\MC(X,\OM)
= \pp D^{\leqslant 0}(X,\OM) \cap \pp D^{\geqslant 0}(X,\OM)$ the
heart of the $t$-structure defined by $p_+$, and we call its objects
$p_+$-perverse sheaves, or dual perverse sheaves. This abelian
category is only artinian.
The $t$-structures defined by $p$ and $p_+$ determine each other
(see \cite[\S 3.3]{BBD}). We have:

\begin{gather}
\label{p+ <= 0}
A \in \pp D^{\leqslant 0}(X,\OM) \Iff 
A \in \p D^{\leqslant 1}(X,\OM) \text{ and }
\p H^1 A \text{ is torsion}
\\
A \in \pp D^{\geqslant 0}(X,\OM) \Iff 
A \in \p D^{\geqslant 0}(X,\OM) \text{ and }
\p H^0 A \text{ is torsion-free}
\\
A \in \p D^{\leqslant 0}(X,\OM) \Iff 
A \in \pp D^{\leqslant 0}(X,\OM) \text{ and }
\pp H^0 A \text{ is divisible}
\\
A \in \p D^{\geqslant 0}(X,\OM) \Iff
A \in \pp D^{\geqslant -1}(X,\OM) \text{ and }
\pp H^{-1} A \text{ is torsion}
\end{gather}

If $A$ is $p$-perverse, then it is also $p_+$-perverse if and only if
$A$ is torsion-free in $\p\MC(X,\OM)$. If $A$ is $p_+$-perverse,
then $A$ is also $p$-perverse if and only if $A$ is divisible in
$\pp\MC(X,\OM)$.
Thus, if $A$ is both $p$- and $p_+$-perverse, then $A$ is
torsion-free in $\p\MC(X,\OM)$ and divisible in $\pp\MC(X,\OM)$.
The modular reduction of a $p$-perverse sheaf $A$ over $\OM$ will be a
perverse over $\FM$ if and only if $A$ is also $p_+$-perverse, and
\emph{vice versa}.

In the following, we will recall
why $(\p  D^{\leqslant 0}, \p  D^{\geqslant 0})$ (resp. the two versions
with $p$ and $p_+$ if $\EM = \OM$) is indeed a $t$-structure on
$D^b_c(X,\EM)$. We refer to \cite{BBD} for more details, however their treatment
of the case $\EM = \OM$ is quite brief, so we give some complements.
The rest of the section is organized as follows.

First, we recall the definition of $t$-categories and their main
properties.  Then we see how they can be combined with torsion
theories.  Afterwards, we recall the notion of recollement of
$t$-categories, stressing on some important properties, such as the
construction of the perverse extensions $\p j_!$, $\p j_{!*}$ and $\p
j_*$ with functors of truncation on the closed part. 
We also study the tops and socles of the extensions
$\p j_!$, $\p j_{!*}$ and $\p j_*$, and show that the intermediate
extension preserves multiplicities.
Then again, we study the
connection with torsion theories.  Already at this point, we have six
possible extensions (the three just mentioned, in the two versions $p$
and $p_+$). 

Then we leave the general context of $t$-structures and recollement
and we focus on perverse sheaves over $\EM = \KM$, $\OM$, $\FM$.
First, we see how the preceding general constructions show that the
definitions of perverse $t$-structures given above actually
give $t$-structures on the triangulated
categories $D^b_c(X,\EM)$, first fixing a stratification, and then
taking the limit. 
Now we have functors $\KM \otimes^\LM_\OM (-)$ and $\FM
\otimes^\LM_\OM (-)$ (it would be nice to treat this situation in an
axiomatic framework, maybe including duality).
We study the connection between modular reduction and truncation.
If we take a complex $A$ over
$\OM$, for each degree we have three places where we can truncate its
reduction modulo $\varpi$, because $\HC^i(\FM A)$ has pieces coming from
$\HC^i_\tors(A)$, $\HC^i_\free(A)$ and $\HC^{i + 1}_\tors(A)$.  So, in
a recollement situation, we have nine natural ways to truncate $\FM A$.

Finally, we introduce decomposition numbers for perverse sheaves, and
particularly in the $G$-equivariant setting. We have in mind, for example,
$G$-equivariant perverse sheaves on the nilpotent cone.

The relation between modular reduction and truncation is really one of
the main technical points. For example, the fact that
the modular reduction does not commute with the intermediate extension
means that the reduction of a simple perverse sheaf will not
necessarily be simple, that is, that we have can have non-trivial
decomposition numbers.

\subsection{$t$-categories}
\label{subsec:ts}

Let us begin by recalling the notion $t$-structure on a triangulated
category, introduced in \cite{BBD}.

\begin{definition}\label{def:ts}
A $t$-category is a triangulated category $\DC$, endowed with two strictly
full subcategories $\DC^{\leqslant 0}$ and $\DC^{\geqslant 0}$, such that,
if we let $\DC^{\leqslant n} = \DC^{\leqslant 0}[-n]$ and
$\DC^{\geqslant n} = \DC^{\geqslant 0}[-n]$, we have
\begin{enumerate}[(i)]
\item For $X$ in $\DC^{\leqslant 0}$ and $Y$ in $\DC^{\geqslant 1}$,
we have $\Hom_\DC(X,Y) = 0$.

\item $\DC^{\leqslant 0} \subset \DC^{\leqslant 1}$ and
$\DC^{\geqslant 0} \supset \DC^{\geqslant 1}$.

\item For each $X$ in $\DC$, there is a distinguished triangle $(A,X,B)$
in $\DC$ with $A$ in $\DC^{\leqslant 0}$ and $B$ in $\DC^{\geqslant 1}$.
\end{enumerate}

We also say that $(\DC^{\leqslant 0},\DC^{\geqslant 0})$ is a $t$-structure on $\DC$.
Its \emph{heart} is the full subcategory $\CC := \DC^{\leqslant 0} \cap \DC^{\geqslant 0}$.
\end{definition}

\begin{proposition}
\label{prop:truncations}
Let $\DC$ be a $t$-category.
\begin{enumerate}[(i)]
\item
\label{prop:truncations:adjunctions}
The inclusion of $\DC^{\leqslant n}$ (resp. $\DC^{\geqslant n}$) in $\DC$
has a right adjoint $\t_{\leqslant n}$ (resp. a left adjoint $\t_{\geqslant n}$).

\item For all $X$ in $\DC$, there is a unique
$d \in \Hom(\t_{\geqslant 1} X, \t_{\leqslant 0} X [1])$
such that the triangle
\[
\t_{\leqslant 0} X \longto X \longto \t_{\geqslant 1} X \elem{d}
\]
is distinguished. Up to unique isomorphism, this is the unique triangle
$(A,X,B)$ with $A$ in $\DC^{\leqslant 0}$ and $B$ in
$\DC^{\geqslant 1}$.

\item Let $a \leqslant b$. Then, for any $X$ in $\DC$, there is a
  unique morphism
  $\t_{\geqslant a} \t_{\leqslant b} X \to \t_{\leqslant b} \t_{\geqslant a} X$
  such that the following diagram is commutative.
\[
\xymatrix{
\t_{\leqslant b} X \ar[r] \ar[d] & X & \t_{\geqslant a} X \ar[l]\\
\t_{\geqslant a} \t_{\leqslant b} X \ar[rr]^\sim&& \t_{\leqslant b}
\t_{\geqslant a} X \ar[u]
}
\]
It is an isomorphism.
\end{enumerate}
\end{proposition}

For example, if $\AC$ is an abelian category and $\DC$ is its derived category,
the natural $t$-structure on $\DC$ is the one for which $\DC^{\leqslant n}$ 
(resp. $\DC^{\geqslant n}$)
is the full subcategory of the complexes $K$ such that $H^i K = 0$ for
$i > n$ (resp. $i < n$). For $K = (K^i, d^i : K^i \to K^{i + 1})$ in $\DC$,
the truncated complex $\t_{\leqslant n} K$ is the subcomplex
$\cdots \to K^{n - 1} \to \Ker d^n \to 0 \to \cdots$ of $K$.
The heart is equivalent to the abelian category $\AC$ we started with.
Note that, in this case, the cone of a morphism $f : A \to B$ between
two objects of $\AC$ is a complex concentrated in degrees $-1$ and $0$.
More precisely, we have $H^{-1}(\Cone f) \simeq \Ker f$ and
$H^0(\Cone f) \simeq \Coker f$. In particular, we have a triangle
$(\Ker f [1],\ \Cone f,\ \Coker f)$.

If we abstract the relations between $\AC$ and $\DC(\AC)$, we get the notion
of admissible abelian subcategory of a triangulated category $\DC$, and a $t$-structure
on $\DC$ precisely provides an admissible abelian subcategory by taking the heart.

More precisely, let $\DC$ be a triangulated category and $\CC$ a full subcategory
of $\DC$ such that $\Hom^i(A,B) := \Hom(A,B[i])$ is zero for $i < 0$ and $A,B$ in $\CC$.
We have the following proposition, which results from the octahedron axiom.

\begin{proposition}
Let $f : X \to Y$ in $\CC$. We can complete $f$ into a distinguished triangle
$(X,Y,S)$. Suppose $S$ is in a distinguished triangle $(N[1],S,C)$ with $N$ and
$C$ in $\CC$. Then the morphisms $N \to S[-1] \to X$ and $Y \to S \to C$,
obtained by composition from the morphisms in the two triangles above, are respectively
a kernel and a cokernel for the morphism $f$ in $\CC$.
\end{proposition}

Such a morphism will be called $\CC$-\emph{admissible}. In a distinguished triangle 
$X\elem{f} Y\elem{g} Z \elem{d}$ on objects in $\CC$, the morphisms
$f$ and $g$ are admissible, $f$ is a kernel of $g$, $g$ is a cokernel of $f$,
and $d$ is uniquely determined by $f$ and $g$. A short exact sequence
in $\CC$ will be called \emph{admissible} if it can be obtained
from a distinguished triangle in $\DC$ by suppressing the degree one morphism.

\begin{proposition}
Suppose $\CC$ is stable by finite direct sums. Then the following conditions are equivalent.
\begin{enumerate}[(i)]
\item $\CC$ is abelian, and its short exact sequences are admissible.

\item Every morphism of $\CC$ is $\CC$-admissible.
\end{enumerate}
\end{proposition}

Now we can state the theorem that says that $t$-structures provide admissible
abelian categories.

\begin{theorem}
The heart $\CC$ of a $t$-category $\DC$ is an admissible abelian subcategory
of $\DC$, stable by extensions.
The functor
$H^0 := \t_{\geqslant 0} \t_{\leqslant 0} \simeq \t_{\leqslant 0}
\t_{\geqslant 0} : \DC \to \CC$ is
a cohomological functor.
\end{theorem}

We have a chain of morphisms
\[
\cdots
\longto \t_{\leqslant i-2}
\longto \t_{\leqslant i-1}
\longto \t_{\leqslant i}
\longto \t_{\leqslant i+1}
\longto \cdots
\]
which can be seen as a ``filtration'' of the identity functor, with
``successive quotients'' the $H^i [-i]$. Thus we have distinguished
triangles:
\[
\t_{\leqslant i-1} \longto \t_{\leqslant i} \longto H^i [-i] \rightsquigarrow
\]
An object $A$ in $\DC$ can be seen as ``made of'' its cohomology
objects $H^i A$ (by successive extensions).
We depict this by the following diagram:
\[
\begin{array}{*{5}{|c}|}
\hline
\T\cdots
& H^{i-1}
& H^{i}
& H^{i+1}
& \cdots
\\
\hline
\multicolumn{1}{|c|}{\cdots}
&\multicolumn{1}{r|}{\SS\t_{\leq i - 1}}
&\multicolumn{1}{r|}{\SS\t_{\leq i}}
&\multicolumn{1}{r|}{\SS\t_{\leq i + 1}}
&\multicolumn{1}{c|}{\cdots}
\end{array}
\]
In the next sections, when we study the interplay between
$t$-structures and other structures (torsion theories, modular
reduction\ldots), we will see refinements of this ``filtration'', and
there will be more complicated pictures.

Now let $\DC_i$ ($i = 1,2$) be two $t$-categories, and let $\e_i : \CC_i \to \DC_i$
denote the inclusion functors of their hearts. Let $T : \DC_1 \to \DC_2$
be a triangulated functor. Then we say that $T$ is
right $t$-exact if $T(\DC_1^{\leqslant 0}) \subset \DC_2^{\leqslant 0}$,
left $t$-exact if $T(\DC_1^{\geqslant 0}) \subset \DC_2^{\geqslant 0}$,
and $t$-exact if it is both left and right $t$-exact.

\begin{proposition}
[Exactness and adjunction properties of the $\p T$]
\label{prop:pT}
\mbox{}
\begin{enumerate}
\item If $T$ is left (resp. right) $t$-exact, then the additive functor
$\p T := H^0 \circ T \circ \e_1$ is left (resp. right) exact.

\item Let $(T^*,T_*)$ be a pair of adjoint triangulated functors, with
$T^* : \DC_2 \to \DC_1$ and $T_* : \DC_1 \to \DC_2$.
Then $T^*$ is right $t$-exact if and only if $T_*$ is left $t$-exact,
and in that case $(\p T^*,\p T_*)$ is a pair of adjoint functors between
$\CC_1$ and $\CC_2$.
\end{enumerate}
\end{proposition}

\subsection{Torsion theories and $t$-structures}
\label{subsec:tt ts}

We will give some variations of known results \cite{HRS}.

\begin{definition}\label{def:tt}
Let $\AC$ be an abelian category. A torsion theory on $\AC$ is a pair
$(\TC, \FC)$ of full subcategories such that
\begin{enumerate}[(i)]
\item \label{eq:hom torsion theory}
for all objects $T$ in $\TC$ and $F$ in $\FC$, we have
\begin{equation*}
\Hom_\AC(T,F) = 0
\end{equation*}
\item \label{eq:ses torsion theory}
for any object $A$ in $\AC$, there are objects $T$ in $\TC$ and $F$ in $\FC$
such that there is a short exact sequence
\begin{equation*}
0 \longto T \longto A \longto F \longto 0
\end{equation*}
\end{enumerate}
\end{definition}

Let us first give some elementary properties of torsion theories.

\begin{proposition}
\label{prop:properties tt}
Let $\AC$ be an abelian category endowed with a torsion theory
$(\TC,\FC)$. Then the following hold:
\begin{enumerate}[(i)]
\item
\label{it:adjoint tt}
The inclusion of $\TC$ (resp. $\FC$) in $\AC$
has a right adjoint  $(-)_\tors : \AC \to \TC$
(resp. a left adjoint $(-)_\free : \AC \to \FC$).

\item
\label{it:ortho tt}
We have
\begin{gather*}
\FC = \TC^\perp = \{F\in\CC \mid \forall T \in \TC,\ \Hom_\CC(T,F) = 0\}
\\
\TC = {}^\perp\FC = \{T\in\CC \mid \forall F \in \FC,\ \Hom_\CC(T,F) = 0\}
\end{gather*}

\item
\label{it:stab tt}
The torsion class $\TC$ (resp. the torsion-free class $\FC$) is
closed under quotients and extensions (resp. under subobjects and
extensions).
\end{enumerate}
\end{proposition}

\begin{definition}
A torsion theory $(\TC,\FC)$ on an abelian category $\AC$ is said to
be \emph{hereditary} (resp. \emph{cohereditary}) if the torsion class $\TC$
(resp. the torsion-free class $\FC$) is closed under subobjects
(resp. under quotients).
\end{definition}

Examples of torsion theories arise with $\OM$-linear abelian categories.

\begin{definition}
\label{def:torsion}
Let $\AC$ be an $\OM$-linear abelian category. An object $A$ in $\AC$ is \emph{torsion}
if $\varpi^N 1_A$ is zero for some $N \in \NM$, and it is
\emph{torsion-free} (resp. \emph{divisible}) if $\varpi.1_A$ is a monomorphism
(resp. an epimorphism).
\end{definition}

\begin{proposition}\label{prop:hom tors div}
Let $\AC$ be an $\OM$-linear abelian category.
\begin{enumerate}[(i)]
\item If $T \in \AC$ is torsion and $F \in \AC$ is torsion-free, then we have
\begin{equation*}
\Hom_\AC(T,F) = 0
\end{equation*}
\item If $Q \in \AC$ is divisible and $T \in \AC$ is torsion, then we have
\begin{equation*}
\Hom_\AC(Q,T) = 0
\end{equation*}
\end{enumerate}
\end{proposition}

\begin{proof}
\emph{(i)}\ Let $f\in \Hom_\AC(T,F)$.
Let $N \in \NM$ such that $\varpi^N 1_T = 0$. Then we have
$(\varpi^N 1_F) \circ f = f \circ (\varpi^N.1_T) = 0$, and consequently $f = 0$,
since $\varpi^N 1_F$ is a monomorphism.

\medskip

\emph{(ii)}\ Let $g\in \Hom_\AC(Q,T)$.
Let $N \in \NM$ such that $\varpi^N 1_T = 0$. Then we have
$g \circ (\varpi^N 1_Q) = (\varpi^N 1_T) \circ g = 0$, and consequently $g = 0$,
since $\varpi^N 1_Q$ is an epimorphism.
\end{proof}

\begin{proposition}
\label{prop:hered}
Let $\AC$ be an $\OM$-linear abelian category. Then subobjects and
quotients of torsion objects are torsion objects.
\end{proposition}

\begin{proof}
Let $T$ be a torsion object in $\AC$. We can choose an integer $N$
such that $\varpi^N 1_T = 0$.

If $i:S\injto T$ is a subobject, then we have
$i\circ (\varpi^N 1_S) = (\varpi^N 1_T) \circ i = 0$, hence
$\varpi^N 1_S = 0$ since $i$ is a monomorphism.
Thus $S$ is torsion.

If $q:T\onto U$ is a quotient, then we have
$(\varpi^N 1_U) \circ q = q\circ (\varpi^N 1_T) = 0$, hence
$\varpi^N 1_U = 0$ since $q$ is an epimorphism.
Thus $U$ is torsion.
\end{proof}

\begin{proposition}
\label{prop:tors div}
Let $A$ be an object in an $\OM$-linear abelian category $\AC$.
\begin{enumerate}[(i)]
\item If $A$ is noetherian, then $A$ has a greatest torsion subobject
$A_\tors$, the quotient $A/A_\tors$ is torsion-free
and $\KM A \simeq \KM A/A_\tors$.

\item If $A$ is artinian, then $A$ has a greatest divisible
subobject $A_\di$, the quotient $A/A_\di$ is a torsion object
and we have $\KM A \simeq \KM A_\di$.
\end{enumerate}
\end{proposition}

\begin{proof}
In the first case,
the increasing sequence $\Ker \varpi^n.1_A$
of subobjects of $A$ must
stabilize, so there is an integer $N$ such that
$\Ker \varpi^n.1_A = \Ker \varpi^N.1_A$ for all $n \geqslant N$.
We set $A_\tors := \Ker \varpi^N.1_A$. This is clearly a torsion object,
since it is killed by $\varpi^N$. Now let $T$ be a torsion subobject
of $A$. It is killed by some $\varpi^k$, and we can assume
$k \geqslant N$. Thus $T \subset \Ker \varpi^k.1_A = \Ker \varpi^N.1_A = A_\tors$.
This shows that $A_\tors$ is the greatest torsion subobject of $A$.
We have
\[
\Ker \varpi.1_{A/A_\tors} = \Ker \varpi^{N+1}.1_A / \Ker \varpi^N.1_A = 0
\]
which shows that $A/A_\tors$ is torsion-free. Applying the exact functor
$\KM \otimes_\OM -$ to the short exact sequence
$0 \to A_\tors \to A \to A/A_\tors \to 0$,
we get $\KM A \simeq \KM A/A_\tors$.

In the second case, the decreasing sequence $\im \varpi^n.1_A$
of subobjects of $A$
must stabilize, so there is an integer $N$ such that
$\im \varpi^n.1_A = \im \varpi^N.1_A$ for all $n \geqslant N$.
We set $A_\di := \im \varpi^N.1_A$. We have
$\im \varpi.1_{A_\di} = \im \varpi^{N+1}.1_A = \im \varpi^N.1_A = A_\di$,
thus $A_\di$ is divisible. We have
\[
\im \varpi^n.1_{A/A_\di} = \im \varpi^n.1_A / \im \varpi^N.1_A = 0
\]
for $n \geqslant N$. Hence $A/A_\di$ is a torsion object.
Applying the exact functor
$\KM \otimes_\OM -$ to the short exact sequence
$0 \to A_\di \to A \to A/A_\di \to 0$,
we get $\KM A_\di \simeq \KM A$.
\end{proof}

\begin{proposition}
\label{prop:torsion theory tors div}
Let $\AC$ be an $\OM$-linear abelian category.
We denote by $\TC$ 
\textup{(}resp. $\FC$, $\QC$\textup{)} 
the full subcategory
of torsion
\textup{(}resp. torsion-free, divisible\textup{)} 
objects in $\AC$.
If $\AC$ is noetherian
\textup{(}resp. artinian\textup{)},
then $(\TC, \FC)$
\textup{(}resp. $(\QC,\TC)$\textup{)}
is an hereditary
\textup{(}resp. cohereditary\textup{)}
torsion theory on $\AC$.
\end{proposition}

\begin{proof}
This follows from Propositions \ref{prop:hom tors div},
\ref{prop:hered} and \ref{prop:tors div}.
\end{proof}

We want to discuss the combination of $t$-structures with torsion theories.

\begin{proposition}
\label{prop:tt ts}
Let $\DC$ be a triangulated category endowed with a $t$-structure
$(\p \DC^{\leqslant 0}, \p \DC^{\geqslant 0})$. Let us denote its
heart by $\CC$, the truncation functors by $\p \t_{\leqslant i}$ and
$\p \t_{\geqslant i}$, and the cohomology functors by $\p H^i : \DC
\to \CC$. Suppose that $\CC$ is endowed with a torsion theory $(\TC,\FC)$.
Then we can define a new $t$-structure
$(\pp \DC^{\leqslant 0}, \pp \DC^{\geqslant 0})$ on $\DC$ by
\[
\pp \DC^{\leqslant 0} = \{A \in \p \DC^{\leqslant 1} \mid  \p H^1(A) \in \TC\}
\]
\[
\pp \DC^{\geqslant 0} = \{A \in \p \DC^{\geqslant 0} \mid  \p H^0(A) \in \FC\}
\]
\end{proposition}

\begin{proof}
Let us check the three axioms for $t$-structures given in Definition \ref{def:ts}.

\medskip

(i)  Let $A \in \pp \DC^{\leqslant 0}$ and $B \in \pp \DC^{\geqslant 1}$.
Then we have
\[
\Hom_\DC(A,B)
= \Hom_\DC(\p \t_{\geqslant 1} A, \p \t_{\leqslant 1} B)
= \Hom_\CC(\p H^1 A, \p H^1 B)
= 0
\]
The first equality follows from the adjunctions of
Proposition \ref{prop:truncations} \eqref{prop:truncations:adjunctions},
since we have
$A \in \pp \DC^{\leqslant 0} \subset \p \DC^{\leqslant 1}$
and $B \in \pp \DC^{\geqslant 1} \subset \p \DC^{\geqslant 1}$.
The second equality follows since $\p \t_{\geqslant 1} A \simeq \p H^1 A [-1]$
and $\p \t_{\leqslant 1} B \simeq \p H^1 B [-1]$.
The last equality follows from the first axiom in the definition of
torsion theories,
since $\p H^1 A \in \TC$ and $\p H^1 B \in \FC$
(see Definition \ref{def:tt} \eqref{eq:hom torsion theory}).

\medskip

(ii) We have
$\pp \DC^{\leqslant 0} \subset \p \DC^{\leqslant 1} \subset \pp \DC^{\leqslant 1}$
and 
$\pp \DC^{\geqslant 0} \supset \p \DC^{\geqslant 1} \supset \pp \DC^{\geqslant 1}$.

\medskip

(iii) Let $A \in \DC$. By
Definition \ref{def:tt} \eqref{eq:ses torsion theory},
there are objects $T \in \TC$ and $F \in \FC$ such that we have
a short exact sequence
\[
0 \longto T \longto \p H^1 A \longto F \longto 0
\]
By \cite[Proposition 1.3.15]{BBD} there is a distinguished triangle
\[
A' \elem{a} A \elem{b} A'' \elem{d} A'[1]
\]
such that $A' \in \p \DC^{\leqslant 1}$
and $A'' \in \p \DC^{\geqslant 1}$,  $\p H^1 A' \simeq T$ and
$\p H^1 A'' \simeq F$, and thus $A' \in \pp \DC^{\leqslant 0}$
and $A'' \in \pp \DC^{\geqslant 1}$.
\end{proof}

We denote by $\CC^+$ the heart of this new $t$-structure, by
$\pp H^n : \DC \to \CC^+$ the new cohomology functors, and by
$\pp \t_{\leqslant n}$, $\pp \t_{\geqslant n}$ the new truncation functors.

We may also use the following notation. For the notions attached to the initial
$t$-structure, we may drop all the $p$, and for the new $t$-structure one
may write $n_+$ instead of $n$, as follows: $(\DC^{\leqslant n_+}, \DC^{\geqslant n_+})$,
$H^{n_+}$, $\t_{\leqslant n_+}$, $\t_{\geqslant n_+}$.

Note that $\CC^+$ is endowed with a torsion theory, namely $(\FC,\TC[-1])$.
We can do the same construction, and we find that $\CC^{++} = \CC [-1]$.
We recover the usual shift of $t$-structures.

We have the following chain of morphisms:
{\small
\[
\cdots
\longto \t_{\leqslant (n - 2)_+}
\longto \t_{\leqslant n - 1}
\longto \t_{\leqslant (n - 1)_+}
\longto \t_{\leqslant n}
\longto \t_{\leqslant n_+}
\longto \t_{\leqslant n + 1}
\longto \cdots
\]
}
and the following distinguished triangles:
\begin{gather}
\label{eq:tri tors}
\t_{\leqslant n} \longto \t_{\leqslant n_+} \longto H^{n + 1}_\tors (-) [- n - 1]
\rightsquigarrow
\\
\label{eq:tri free}
\t_{\leqslant n_+} \longto \t_{\leqslant n + 1} \longto H^{n + 1}_\free (-) [- n - 1]
\rightsquigarrow
\end{gather}
This follows from \cite[Prop. 1.3.15]{BBD}, which
is proved using the octahedron axiom.
These triangles can be read off the following diagram:

\[
\begin{array}{|*{7}{c|}}
\hline
\multicolumn{2}{|c|}{\T H^{n-1}}
&\multicolumn{2}{c|}{H^n}
&\multicolumn{2}{c|}{H^{n+1}}
&\multicolumn{1}{c|}{\cdots}
\\
\hline
\T\B
H^{n-1}_\tors
&H^{n-1}_\free
&H^{n}_\tors
&H^{n}_\free
&H^{n+1}_\tors
&H^{n+1}_\free
&H^{n+2}_\tors
\\
\hline
\multicolumn{1}{|c|}{\T \cdots}
&\multicolumn{2}{c|}{H^{(n-1)_+}}
&\multicolumn{2}{c|}{H^{n_+}}
&\multicolumn{2}{c|}{H^{(n+1)_+}}
\\
\hline
\multicolumn{1}{|c|}{\cdots}
&\multicolumn{1}{r|}{\SS \t_{\leq n - 1}}
&\multicolumn{1}{r|}{\SS \t_{\leq (n - 1)_+}}
&\multicolumn{1}{r|}{\SS \t_{\leq n}}
&\multicolumn{1}{r|}{\SS \t_{\leq n_+}}
&\multicolumn{1}{r|}{\SS \t_{\leq n + 1}}
&\multicolumn{1}{c|}{\cdots}
\end{array}
\]

If $\DC$ is an $\OM$-linear $t$-category, then its heart
$\CC$ is also $\OM$-linear.  If $\CC$ is noetherian (resp. artinian),
then it is naturally endowed with a torsion theory by Proposition
\ref{prop:torsion theory tors div}, and the preceding considerations
apply.

Assume, for example, that $\CC$ is noetherian, endowed with the
torsion theory $(\TC,\FC)$, where $\TC$ (resp. $\FC$) is the full
subcategory of torsion (resp. torsion-free) objects in $\CC$.
For $L$ is $\FC$, $\varpi 1_L$ is a monomorphism in $\CC$,
and we have a short exact sequence in $\CC$
\[
0 \longto L \elem{\varpi 1_L} L \longto \Coker_\CC \varpi 1_L
\longto 0
\]
Since $\CC$ is an admissible abelian subcategory of $\DC$, this short
exact sequence comes from a distinguished triangle in $\DC$
\[
L \elem{\varpi 1_L} L \longto \Coker_\CC \varpi 1_L \rtordu
\]
Rotating it (by the TR 2 axiom), we get a distinguished triangle
\[
\Coker_\CC \varpi 1_L [-1] \longto L \elem{\varpi 1_L} L \rtordu
\]
all whose objects are in $\CC^+$. Since this abelian subcategory is
also admissible, we have the following short exact sequence in $\CC^+$
\[
0 \longto \Coker_\CC \varpi 1_L [-1] \longto L \elem{\varpi 1_L} L 
\longto 0
\]
showing that $\varpi 1_L$ is an epimorphism in $\CC^+$ (that is, $L$
is divisible in $\CC^+$), and that
$\Ker_{\CC^+} \varpi 1_L = \Coker_\CC \varpi 1_L [-1]$.

\begin{example}
\label{ex:point}
Let us consider $\DC = D^b_c(\OM)$, the full subcategory of the
bounded derived category of $\OM$-modules, whose objects are the
complexes all of whose cohomology groups are finitely generated over
$\OM$. We can take the natural $t$-structure $(\DC^{\leqslant
  0},\DC^{\geqslant 0})$. The heart $\CC$ is then the abelian category
of finitely generated $\OM$-modules (we identify such a module with
the corresponding complex concentrated in degree zero). The category $\CC$ is
noetherian but not artinian: the object $\OM$ has an infinite
decreasing sequence of subobjects $(\mG^n)$. In $\CC$, it is a
torsion-free object: $\varpi^n 1_\OM$ is a monomorphism in $\CC$, with
cokernel $\OM/\mG^n$.

Now, we can look at $\OM$ as an object of the abelian category $\CC^+$
obtained as above. Then $\OM$ is a divisible object in $\CC^+$:
$\varpi^n 1_\OM$ is an epimorphism, with kernel
$\OM/\mG^n [-1]$. This provides an infinite increasing sequence of
subobjects of $\OM$ in $\CC^+$, showing that $\CC^+$ is not noetherian.
\end{example}

\begin{remark}
The preceding example is just about perverse sheaves
on a point, for the perversities $p$ and $p_+$.
\end{remark}

\subsection{Recollement}
\label{subsec:recollement}

The recollement (gluing) construction consists roughly in a way to construct
a $t$-structure on some derived category of sheaves on a topological space
(or a ringed topos) $X$, given $t$-structures on derived categories of sheaves on
$U$ and on $F$, where $j : U \to X$ is an open subset of $X$, and
$i : F \to X$ its closed complement. This can be done in a very general axiomatic
framework \cite[\S 1.4]{BBD}, which can be applied to both the complex topology
and the étale topology. The axioms can even be applied to non-topological
situations, for example for representations of algebras. Let us recall the definitions 
and main properties of the recollement procedure.

So let $\DC$, $\DC_U$ and $\DC_F$ be three triangulated categories, and let
$i_* : \DC_F \to \DC$ and $j^* : \DC \to \DC_U$ be triangulated functors.
It is convenient to set $i_! = i_*$ and $j^! = j^*$. We assume that
the following conditions are satisfied.

\begin{assumption}[Recollement situation]
\label{ass:recollement}
\mbox{}
\begin{enumerate}[(i)]
\item $i_*$ has triangulated left and right adjoints, denoted by $i^*$ and
  $i^!$ respectively.

\item $j^*$ has triangulated left and right adjoints, denoted by $j_!$
  and $j_*$ respectively.

\item We have $j^* i_* = 0$. By adjunction, we also have $i^* j_! = 0$
  and $i^! j_* = 0$. Moreover, for $A$ in $\DC_F$ and $B$ in $\DC_U$, we
  have
\[
\Hom(j_! B, i_* A) = 0 \text{ and } \Hom(i_* A, j_* B) = 0
\]

\item For all $K$ in $\DC$, there exists $d : i_* i^* K \to j_! j^* K
  [1]$ (resp. $d : j_* j^* K \to i_* i^! K [1]$), necessarily unique,
  such that the triangle $j_! j^* K \to K \to i_* i^* K \map{d}$
  (resp. $i_* i^! K \to K \to j_* j^* K \map{d}$) is distinguished.

\item The functors $i_*$, $j_!$ and $j_*$ are fully faithful: the
  adjunction morphisms $i^* i_* \to \id \to i^! i_*$ and $j^* j_* \to
  \id \to j^* j_!$ are isomorphisms.
\end{enumerate}
\end{assumption}

Whenever we have a diagram
\begin{equation}\label{eq:recollement setup}
\xymatrix{
\DC_F
\ar[r]^{i_*}
&
\DC
\ar@<-3ex>[l]_{i^*}
\ar@<3ex>[l]_{i^!}
\ar[r]^{j^*}
&
\DC_U
\ar@<-3ex>[l]_{j_!}
\ar@<3ex>[l]_{j_*}
}
\end{equation}
such that the preceding conditions are satisfied, we say that we are
in a situation of recollement.

Note that for each recollement situation, there is a dual recollement
situation on the opposite derived categories. Recall that the opposite
category of a triangulated category $\TC$ is also triangulated, with
translation functor $[-1]$, and distinguished triangles the triangles
$(Z,Y,X)$, where $(X,Y,Z)$ is a distinguished triangle in $\TC$.
One can check that the conditions in Assumption \ref{ass:recollement} are satisfied
for the following diagram, where the roles of $i^*$ and $i^!$ (resp.
$j_!$ and $j_*$) have been exchanged.

\begin{equation}\label{eq:dual recollement setup}
\xymatrix{
\DC_F^\op
\ar[r]^{i_*}
&
\DC^\op
\ar@<-3ex>[l]_{i^!}
\ar@<3ex>[l]_{i^*}
\ar[r]^{j^*}
&
\DC_U^\op
\ar@<-3ex>[l]_{j_*}
\ar@<3ex>[l]_{j_!}
}
\end{equation}
We can say that there is a ``formal duality'' in the axioms of a
recollement situation, exchanging the symbols $!$ and $*$. Note that,
in the case of $D^b_c(X,\EM)$, the duality $\DC_{X,\EM}$ really
exchanges these functors.

If $\UC \map{u} \TC \map{q} \VC$ is a sequence of triangulated
functors between triangulated categories such that $u$ identifies
$\UC$ with a thick subcategory of $\TC$, and $q$ identifies $\VC$ with
the quotient of $\TC$ by the thick subcategory $u(\UC)$, then we say
that the sequence $0 \to \UC \map{u} \TC \map{q} \VC \to 0$ is exact.

\begin{proposition}
The sequences
\begin{gather*}
0 \longfrom \DC_\FC \leftelem{i^*} \DC \leftelem{j_!} \DC_U \longfrom 0\\
0 \longto \DC_\FC \elem{i_*} \DC \elem{j^*} \DC_U \longto 0\\
0 \longfrom \DC_\FC \leftelem{i^!} \DC \leftelem{j_*} \DC_U \longfrom 0
\end{gather*}
are exact.
\end{proposition}

Suppose we are given a $t$-structure
$(\DC_U^{\leqslant 0},\DC_U^{\geqslant 0})$ on $\DC_U$,
and a $t$-structure
$(\DC_F^{\leqslant 0},\DC_F^{\geqslant 0})$ on $\DC_F$.
Let us define
\begin{gather}
\label{eq:recollement def <= 0}
\DC^{\leqslant 0} := \{K \in \DC \mid j^*K \in \DC_U^{\leqslant 0}
\text{ and } i^*K \in \DC_F^{\leqslant 0} \} \\
\label{eq:recollement def >= 0}
\DC^{\geqslant 0} := \{K \in \DC \mid j^*K \in \DC_U^{\geqslant 0}
\text{ and } i^!K \in \DC_F^{\geqslant 0} \}
\end{gather}

\begin{theorem}\label{th:recollement}
With the preceding notations, $(\DC^{\leqslant 0}, \DC^{\geqslant 0})$
is a $t$-structure on $\DC$.
\end{theorem}

We say that it is obtained from those on
$\DC_U$ and $\DC_F$ by \emph{recollement} (gluing).

Now suppose we are just given a $t$-structure on $\DC_F$. Then we can
apply the recollement procedure to the degenerate $t$-structure
$(\DC_U,0)$ on $\DC_U$ and to the given $t$-structure on $\DC_F$. The
functors $\t_{\leqslant n}$ ($n \in \ZM$) relative to the $t$-structure obtained on
$\DC$ will be denoted $\t^F_{\leqslant n}$. The functor
$\t^F_{\leqslant n}$ is right adjoint to the inclusion of the full
subcategory of $\DC$ whose objects are the $X$ such that $i^* X$ is in
$\DC_F^{\leqslant n}$. We have a distinguished triangle
$(\t^F_{\leqslant n} X, X, i_* \t_{> n} i^* X)$. The $H^n$ cohomology
functors for this $t$-structure are the $i_* H^n i^*$.
Thus we have a chain of morphisms:
\begin{equation}
\label{eq:chain ptF}
\cdots \longto \t^F_{\leqslant n - 1}
\longto \t^F_{\leqslant n} \longto \t^F_{\leqslant n + 1}
\longto \cdots
\end{equation}
and distinguished triangles:
\begin{equation}
\label{eq:tri ptF}
\t^F_{\leqslant n} \longto \t^F_{\leqslant n + 1}
\longto i_* H^{n+1} i^* [- n - 1]
\rightsquigarrow
\end{equation}
We summarize this by the following diagram:
\[
\begin{array}{*{5}{|c}|}
\hline
\T\cdots
& i_* H^{n-1} i^*
& i_* H^{n} i^*
& i_* H^{n+1} i^*
& \cdots
\\
\hline
\multicolumn{1}{|c|}{\cdots}
&\multicolumn{1}{r|}{\SS\t^F_{\leq n - 1}}
&\multicolumn{1}{r|}{\SS\t^F_{\leq n}}
&\multicolumn{1}{r|}{\SS\t^F_{\leq n + 1}}
&\multicolumn{1}{c|}{\cdots}
\end{array}
\]
One has to keep in mind, though, that this $t$-structure is
degenerate, so an object should not be thought as ``made of'' its
``successive quotients'' $i_* H^n i^*$ (an object in $j_! \DC_U$ will
be in $\DC_U^{\leq n}$ for all $n$). 

Dually, one can define the functor $\t^F_{\geqslant n}$ using the
degenerate $t$-structure $(0,\DC_U)$ on $\DC_U$. It is left adjoint to
the inclusion of $\{X \in \DC \mid i^! X \in \DC_F^{\geqslant n}\}$ in
$\DC$, we have distinguished triangles
$(i_* \t_{< n} i^! X, X, \t^F_{\geqslant n} X)$,
and the $H^n$ are the $i_* H^n i^!$.

Similarly, if we are just given a $t$-structure on $\DC_U$, and if we
endow $\DC_F$ with the degenerate $t$-structure $(\DC_F, 0)$
(resp. $(0, \DC_F)$), we can define a $t$-structure on $\DC$ for which
the functors $\t_{\leqslant n}$ (resp. $\t_{\geqslant n}$), denoted by 
$\t^U_{\leqslant n}$ (resp. $\t^U_{\geqslant n}$), yield distinguished
triangles $(\t^U_{\leqslant n}, X, j_* \t_{> n} j^* X)$
(resp. $(j_! \t_{< n} j^* X, X, \t^U_{\geqslant n} X)$),
and for which the $H^n$ functors are the $j_* H^n j^*$
(resp. $j_! H^n j^*$).

Moreover, we have
\begin{equation}
\label{eq:tFtU}
\t_{\leqslant n} = \t_{\leqslant n}^F \t_{\leqslant n}^U
\text{ and }
\t_{\geqslant n} = \t_{\geqslant n}^F \t_{\geqslant n}^U
\end{equation}

An \emph{extension} of an object $Y$ of $\DC_U$ is an object $X$ of
$\DC$ endowed with an isomorphism $j^* X \isom Y$. Such an isomorphism
induces morphisms $j_! Y \to X \to j_* Y$ by adjunction. If an
extension $X$ of $Y$ is isomorphic, as an extension, to
$\t_{\geqslant n}^F j_! Y$ (resp. $\t_{\leqslant n}^F j_* Y$), then
the isomorphism is unique, and we just write $X = \t_{\geqslant n}^F j_! Y$
(resp. $\t_{\leqslant n}^F j_* Y$).

\begin{proposition}
Let $Y$ in $\DC_U$ and $n$ an integer. There is, up to unique
isomorphism, a unique extension $X$ of $Y$ such that $i^* X$ is in
$\DC_F^{\leqslant n - 1}$ and
$i^! X$ is in $\DC_F^{\geqslant n + 1}$. It is
$\t^F_{\leqslant n - 1} j_* Y$, and this extension of $Y$ is
canonically isomorphic to $\t^F_{\geqslant n + 1} j_! Y$.
\end{proposition}

Let $\DC_m$ be the full subcategory of $\DC$ consisting of the objects
$X$ such that $i^* X \in \DC_F^{\leqslant n - 1}$ and
$i^! X \in \DC_F^{\leqslant n + 1}$. The functor $j^*$ induces an
equivalence $\DC_m \to \DC_U$, with quasi-inverse
$\t^F_{\leqslant n - 1} j_* = \t^F_{\geqslant n + 1} j_!$, which will
be denoted $j_{!*}$.

Let $\CC$, $\CC_U$ and $\CC_F$ denote the hearts of the $t$-categories
$\DC$, $\DC_U$ and $\DC_F$. We will use the notation $\p T$ of
Proposition \ref{prop:pT}, where $T$ is one of the functors of the
recollement diagram \eqref{eq:recollement setup}.
By definition of the $t$-structure of $\DC$, $j^*$ is $t$-exact, $i^*$
is right $t$-exact, and $i^!$ is left $t$-exact. Applying
Proposition \ref{prop:pT}, we get the first two points of the following
proposition.

\begin{proposition}
\label{prop:properties 6 functors}
The functors $\p j_!$, $\p j^*$, $\p j_*$, $\p i^*$, $\p i_*$, $\p i^!$
have the following properties:
\begin{enumerate}[(i)]
\item
\label{it:adj p i_*} 
The functor $\p i_*$ has left and right adjoints $\p i^*$ and $\p i^!$.
Hence $\p i_*$ is exact, $\p i^*$ is right exact and $\p i^!$ is left
exact.

\item
\label{it:adj p j^*}
The functor $\p j^*$ has left and right adjoints $\p j_!$ and $\p j_*$.
Hence $\p j^*$ is exact, $\p j_!$ is right exact and $\p j_*$ is left
exact.


\item
\label{it:comp 6f}
The compositions $\p j^* \p i_*$, $\p i^* \p j_!$ and $\p i^! \p j_*$ are
zero. For $A$ in $\CC_F$ and $B$ in $\CC_U$, we have
\[
\Hom(\p j_! B, \p i_* A) = 0 \text{ and } \Hom(\p i_* A, \p j_* B) = 0
\]

\item
\label{it:ex seq}
For any object $A$ in $\CC$, we have exact sequences
\begin{gather}
\label{eq:ex seq a}
0 \longto \p i_* H^{-1} i^* A \longto \p j_! \p j^* A \longto A
\longto \p i_* \p i^* A \longto 0\\
\label{eq:ex seq b}
0 \longto \p i_* \p i^! A \longto A \longto \p j_* \p j^* A \longto \p
i_* H^1 i^! A \longto 0
\end{gather}

\item
\label{it:fully faithful}
The functors $\p i_*$, $\p j_!$ and $\p j_*$ are fully faithful: the
adjunction morphisms $\p i^* \p i_* \to \Id \to \p i^! \p i_*$
and $\p j^* \p j_* \to \Id \to \p j^* \p j_!$ are isomorphisms.

\item
\label{it:thick subcat}
The essential image of the fully faithful functor $\p i_*$ is a thick subcategory of $\CC$.
For any object $A$ in $\CC$, $\p i_* \p i^* A$ is the largest quotient
of $A$ in $\p i_* \CC_F$, and $\p i_* \p i^! A$ is the largest subobject of
$A$ in $\p i_* \CC_F$.

\item
\label{it:quotient}
  The functor $\p j^*$ identifies $\CC_U$ with the quotient of
  $\CC$ by the thick subcategory $\p i_* \CC_F$.
\end{enumerate}
\end{proposition}

Since $j^*$ is a quotient functor of triangulated categories, the
composition of the adjunction morphisms $j_! j^* \to \id \to j_* j^*$
comes from a unique morphism of functors $j_! \to j_*$. Applying
$j^*$, we get the identity automorphism of the identity functor.

Similarly, since the functor $\p j^*$ is a quotient functor of abelian
categories, the composition of the adjunction morphisms $\p j_! \p j^*
\to \id \to \p j_* \p j^*$ comes from a unique morphism of functors
$\p j_! \to \p j_*$. Applying $\p j^*$, we get the identity automorphism
of the identity functor.

Let $\p j_{!*}$ be the image of $\p j_!$ in $\p j_*$. We have a factorization
\begin{equation}
j_! \longto \p j_! \longto \p j_{!*} \longto \p j_* \longto j_*
\end{equation}

The following characterization of the functors $\p j_!$, $\p j_{!*}$
and $\p j_*$ will be very useful.

\begin{proposition}
\label{prop:characterization}
We have
\begin{alignat}{5}
\p j_! &\ =\ & \t^F_{\geqslant 0}\ j_! &\ =\ & \t^F_{\leqslant -2}\ j_*\\
\p j_{!*} &\ =\ & \t^F_{\geqslant 1}\ j_! &\ =\ & \t^F_{\leqslant -1}\ j_*\\
\p j_* &\ =\ & \t^F_{\geqslant 2}\ j_! &\ =\ & \t^F_{\leqslant 0}\ j_*
\end{alignat}
\end{proposition}

So \eqref{eq:chain ptF} and \eqref{eq:tri ptF} now
read: we have a chain of morphisms:
\[
\p j_! \longto \p j_{!*} \longto \p j_*
\]
and distinguished triangles:
\begin{gather}
\label{tri:! !*}
\p j_! \longto \p j_{!*} \longto i_* H^{-1} i^* j_* [1] \rightsquigarrow\\
\label{tri:!* *}
\p j_{!*} \longto \p j_* \longto i_* H^0 i^* j_* \rightsquigarrow
\end{gather}

In other words,
for $A$ in $\CC$, the kernel and cokernel of $\p j_! A \to \p j_* A$
are in $\p i_* \CC_F$, and we have the following Yoneda splice of
two short exact sequences:
\begin{equation*}
\xymatrix@C=.5cm @R=.1cm{
0\ar[r] &
i_* H^{-1} i^* j_* A \ar[r] &
\p j_! A \ar[dr] \ar[rr] &&
\p j_* A \ar[r] &
i_* H^0 i^* j_* A \ar[r] &
0\\
&&& \p j_{!*} A \ar[ur] \ar[dr]\\
&& 0 \ar[ur] && 0
}
\end{equation*}

\begin{corollary}
For $A$ in $\CC_U$, $\p j_{!*} A$ is the unique extension $X$ of $A$ in
$\DC$ such that $i^* X$ is in $\DC_F^{\leqslant -1}$ and $i^! X$ is in
$\DC_F^{\geqslant 1}$. Thus it is the unique extension of $A$ in $\CC$
with no non-trivial subobject or quotient in $\p i_* \CC_F$.

Similarly, $\p j_! A$ (resp. $\p j_* A$) is the unique extension $X$
of $A$ in $\DC$ such that $i^* X$ is in $\DC_F^{\leqslant -2}$
(resp. $\DC_F^{\leqslant 0}$) and $i^! X$ is in $\DC_F^{\geqslant 0}$
(resp. $\DC_F^{\geqslant 2}$). In particular, $\p j_! A$ (resp. $\p
j_* A$) has no non-trivial quotient (resp. subobject) in $\p i_* \CC_F$.
\end{corollary}

Building on the preceding results, it is now easy to get the following
description of the simple modules in $\CC$.

\begin{proposition}
The simple objects in $\CC$ are the $\p i_* S$, with $S$ simple in
$\CC_F$, and the $\p j_{!*} S$, for $S$ simple in $\CC_U$.
\end{proposition}

Let $\SC$ (resp. $\SC_U$, $\SC_F$) denote the set of (isomorphisms
classes of) simple objects in $\CC$ (resp. $\CC_U$, $\CC_F$). So we have
$\SC = \p j_{!*} \SC_U \cup \p i_* \SC_F$.  Let us assume that $\CC$,
$\CC_U$ and $\CC_F$ are noetherian and artinian, so that the multiplicities of the
simple objects and the notion of composition length are well-defined.
Thus, if $B$ is an object in $\CC$, then we have the following
relation in the Grothendieck group $K_0(\CC)$:
\begin{equation}
\label{eq:def mult}
[B] = \sum_{T \in\SC} [B : T] \cdot [T]
\end{equation}

We will now show that $\p j_{!*}$ preserves multiplicities.

\begin{proposition}
\label{prop:IC mult}
If $B$ is an object in $\CC$, then we have
\begin{equation}\label{eq:j^*}
[B : \p j_{!*} S] = [j^* B : S]
\end{equation}
for all simple objects $S$ in $\CC_U$. 
In particular, if $A$ is an object in $\CC_U$, then we have
\begin{equation}\label{eq:j_!*}
[\p j_! A : \p j_{!*} S] = [\p j_{!*}A : \p j_{!*}S]
= [\p j_* A : \p j_{!*} S] = [A:S]
\end{equation}
\end{proposition}

\begin{proof}
The functor $j^*$ is exact, and sends a simple object $T$ on a simple
a simple object if $T\in \p j_{!*}\SC_U$, or on zero if $T \in \p i_* \SC_F$.
Moreover, it sends non-isomorphic simple objects in $\p j_{!*}\SC_U$ on
non-isomorphic simple objects in $\SC_U$.
Thus, applying $j^*$ to the relation \eqref{eq:def mult}, we get
\[
[j^* B] 
= \sum_{S \in \SC_U} [j^* B : S] \cdot [S]
= \sum_{S \in \SC_U} [B : \p j_{!*} S] \cdot [S]
\]
hence \eqref{eq:j^*}, and \eqref{eq:j_!*} follows.
\end{proof}

\begin{proposition}
\label{prop:inj surj}
The functor $\p j_{!*}$ preserves monomorphisms and epimorphisms.
\end{proposition}

\begin{proof}
Let $u:A\to B$ be a monomorphism in $\CC_U$.
Let $K$ be the kernel of the morphism
$\p j_{!*}u : \p j_{!*}A \to \p j_{!*}B$ in $\CC$.
Since this morphism becomes a monomorphism after applying
$\p j^*$ (restriction to $U$), $K$ is in $\p i_* \CC_F$. But $K$
is a subobject of $\p j_{!*}A$, which has no non-trivial
subobject in $\p i_* \CC_F$. Hence $K = 0$ and $\p j_{!*}u$ is
a monomorphism.

Dually, let $v:A\to B$ be an epimorphism in $\CC_U$.
Let $C$ be the cokernel of the morphism
$\p j_{!*}v : \p j_{!*}A \to \p j_{!*}B$ in $\CC$.
Since this morphism becomes an epimorphism after applying
$\p j^*$ (restriction to $U$), $C$ is in $\p i_* \CC_F$. But $C$
is a quotient of $\p j_{!*}B$, which has no non-trivial
quotient in $\p i_* \CC_F$. Hence $C = 0$ and $\p j_{!*}v$ is
an epimorphism.
\end{proof}

\begin{proposition}
\label{prop:top socle}
Let $A$ be an object of $\CC_U$. Then we have
\begin{gather*}
\Soc \p j_{!*} A
\simeq \Soc \p j_* A
\simeq \p j_{!*} \Soc A\\
\Top \p j_! A
\simeq \Top \p j_{!*} A
\simeq \p j_{!*} \Top A
\end{gather*}
\end{proposition}

\begin{proof}
By definition, $\p j_{!*} A$ is a subobject of $\p j_* A$.
Taking socles, we get $\Soc \p j_{!*} A \subset \Soc \p j_* A$
as subobjects of $\p j_* A$.

By applying the exact functor $\p j^*$ to the monomorphism
$\Soc \p j_* A \subset \p j_* A$, we get a monomorphism
$\p j^* \Soc \p j_* A \subset A$.
But $\p j^* \Soc \p j_* A$ is semisimple, so we get
$\p j^* \Soc \p j_* A \subset \Soc A$ as subobjects of $A$.
Thus, by Proposition \ref{prop:inj surj},
we have $\p j_{!*} \p j^* \Soc \p j_* A \subset \p j_{!*} \Soc A$
as subobjects of $\p j_* A$.
Now, we have $\p j_{!*} \p j^* \Soc \p j_* A = \Soc \p j_* A$
because $\Soc \p j_* A$ is a semisimple object with no simple
constituent in $\p i_* \CC_F$. Hence $\Soc \p j_* A \subset \p j_{!*} \Soc A$
as subobjects of $\p j_* A$.

By Proposition \ref{prop:inj surj}, if we apply the functor $\p
j_{!*}$ to the monomorphism $\Soc A \subset A$, we get a monomorphism
$\p j_{!*}\Soc A \subset \p j_{!*} A$.
But  $\p j_{!*}\Soc A$ is semisimple, so we get
$\p j_{!*}\Soc A \subset \Soc \p j_{!*} A$
as subobjects of $\p j_* A$.

This proves the first relation, and the second one is dual.
\end{proof}

\begin{proposition}
\label{prop:j_!* fully faithful}
The functor $\p j_{!*}$ is fully faithful.
\end{proposition}

\begin{proof}
Let $A$ and $B$ be two objects in $\CC_U$.
Applying the left exact functor $\Hom_\CC(-,\p j_{!*} B)$
to the short exact sequence
\[
0 \longto
\p i_* \p i^! A \longto \p j_! A \longto \p j_{!*} A
\longto 0
\]
we get an exact sequence
\[
0 \longto \Hom_\CC(\p j_{!*} A, \p j_{!*} B)
\longto \Hom_\CC(\p j_! A, \p j_{!*} B)
\longto \Hom_\CC(\p i_* \p i^! A, \p j_{!*} B)
\]
Since $\p j_{!*} B$ has no non-trivial subobject in $\p i_* \CC_F$, we
deduce that 
\[
\Hom_\CC(\p i_* \p i^! A, \p j_{!*} B) = 0
\]
and thus we have an isomorphism
\[
\Hom_\CC(\p j_{!*} A, \p j_{!*} B)
\elem{\sim} \Hom_\CC(\p j_! A, \p j_{!*} B)
\]

Similarly, applying the left exact functor $\Hom_\CC(\p j_! A,-)$
to the short exact sequence
\[
0 \longto \p j_{!*} B \longto \p j_* B \longto \p i_* \p i^* B 
\longto 0
\]
and using the relation $\Hom_\CC(\p j_! A, \p i_* \p i^* B) = 0$,
we get an isomorphism 
\[
\Hom_\CC(\p j_! A, \p j_{!*} B) \elem{\sim}
\Hom_\CC(\p j_! A, \p j_* B)
\]
Now using 
Proposition \ref{prop:properties 6 functors} \eqref{it:adj p j^*}
and \eqref{it:fully faithful},
we see that the latter
Hom space is isomorphic to $\Hom_{\CC_U}(A,B)$.
Hence $\p j_{!*}$ is fully faithful.
\end{proof}

\subsection{Torsion theories and recollement}
\label{subsec:tt recollement}

We will see now how to glue torsion theories in the recollement procedure.

\begin{proposition}
\label{prop:tt recollement}
Suppose we are in a recollement situation as in Subsection
\ref{subsec:recollement}, and that we are given
torsion theories $(\TC_F, \FC_F)$ and
$(\TC_U, \FC_U)$ of $\CC_F$ and $\CC_U$. Then we can define a torsion theory
$(\TC, \FC)$ on $\CC$ by
\begin{gather}
\TC = \{ T \in \CC \mid \p i^* T \in \TC_F \text{ and } \p j^* T \in \TC_U \}\\
\FC = \{ L \in \CC \mid \p i^! L \in \FC_F \text{ and } \p j^* L \in \FC_U \}
\end{gather}
\end{proposition}

Let us begin by some lemmas.

\begin{lemma}
\label{lemma:stab rec tt}
The subcategory $\TC$ (resp. $\FC$) of $\CC$
is closed under quotients and extensions
(resp. under subobjects and extensions).
\end{lemma}

\begin{proof}
Let us consider a short exact sequence in $\CC$
\[
0 \longto S \longto A \longto Q \longto 0
\]

Applying the functors $\p i^*$, $\p j^*$ and $\p i^!$, we get three
exact sequences:
\begin{gather}
\label{ses:p i^*}
\p i^* S \longto \p i^* A \longto \p i^* Q \longto 0
\\
\label{ses:p j^*}
0 \longto \p j^* S \longto \p j^* A \longto \p j^* Q \longto 0
\\
\label{ses:p i^!}
0 \longto \p i^! S \longto \p i^! A \longto \p i^! Q
\end{gather}

Let us first assume that $A$ is in $\TC$, and let us show that $Q$ is
also in $\TC$. We have to show that $\p i^*Q$ is in $\TC_F$ and that
$\p j^* Q$ is in $\TC_U$. This follows from Proposition
\ref{prop:properties tt}, since $\p i^*Q$ is a quotient of
$\p i^*A$ and $\p j^* Q$ is quotient of $\p j^* A$.

Secondly, suppose that $S$ and $Q$ are in $\TC$, and let us show that
$A$ is also in $\TC$. We have to show that $\p i^*A$ is in $\TC_F$ and that
$\p j^* A$ is in $\TC_U$. This follows also from Proposition
\ref{prop:properties tt}, since $\p i^* A$ is an extension
of $\p i^* Q$ by a quotient of $\p i^* S$, and 
$\p j^* A$ is an extension of $\p j^* Q$ by $\p j^* S$.

The proofs for the statements about $\FC$ are dual.

\end{proof}

\begin{lemma}
\label{lemma:direct images rec tt}
We have
\begin{gather*}
\p i_*(\TC_F) \subset \TC\qquad
\p j_!(\TC_U)\subset \TC\qquad
\p j_{!*}(\TC_U) \subset \TC\\
\p i_*(\FC_F) \subset \FC\qquad
\p j_*(\FC_U) \subset \FC\qquad
\p j_{!*}(\FC_U) \subset \FC
\end{gather*}
\end{lemma}

\begin{proof}
This follows from Proposition \ref{prop:properties 6 functors}
\eqref{it:comp 6f} and \eqref{it:fully faithful},
the definition of $(\TC, \FC)$, the definition of $\p j_{!*}$, and
Lemma \ref{lemma:stab rec tt}.
\end{proof}

\begin{lemma}
\label{lemma:Hom(T,L) = 0}
If $T \in \TC$ and $L \in \FC$, then we have
$\Hom_\CC(T,L) = 0$.
\end{lemma}

\begin{proof}
By Proposition
\ref{prop:properties 6 functors} \eqref{it:ex seq},
we have an exact sequence \eqref{eq:ex seq a}
\[
\p j_! \p j^* T \longto T \longto \p i_* \p i^* T \longto 0
\]
Applying the functor $\Hom_{\CC}(-,L)$, which is left exact,
we get an exact sequence
\[
0 \longto \Hom_{\CC}(\p i_* \p i^* T, L)
\longto \Hom_{\CC}(T, L) 
\longto \Hom_{\CC}(\p j_! \p j^* T, L) 
\]
By the adjunctions of Proposition
\ref{prop:properties 6 functors} \eqref{it:adj p i_*} and
\eqref{it:adj p j^*}, this
becomes
\[
0 \longto \Hom_{\CC_F}(\p i^* T, \p i^! L)
\longto \Hom_{\CC}(T, L) 
\longto \Hom_{\CC_U}(\p j^* T, \p j^* L) 
\]
Now, we have  $\Hom_{\CC_F}(\p i^* T, \p i^! L) = 0$ because
$\p i^* T \in \TC_F$ and $\p i^! L \in \FC_F$,
and similarly $\Hom_{\CC_U}(\p j^* T, \p j^* L) = 0$ because
$\p j^* T \in \TC_U$ and $\p j^* L \in \FC_U$. Thus
$\Hom_{\CC}(T, L) = 0$.
\end{proof}

We are now ready to prove the Proposition. To check the second axiom for
torsion theories, the idea is the following: given an object $A$, of
$\CC$, we construct a filtration $0 \subset S \subset B \subset A$
where $S$ is in $\TC$, $A/B$ is in $\FC$, and $M := B/S$ is in 
$\p i_* \CC_F$. Then we use the torsion theory on $\CC_F$ to cut $M$
into a torsion part and a torsion-free part. Taking the inverse image
in $B$ of the torsion part of $M$, we get the torsion subobject of
$A$. Now let us give the details.

\begin{proof}[Proof of Proposition \ref{prop:tt recollement}]
The first axiom for torsion theories has been checked in Lemma
\ref{lemma:Hom(T,L) = 0}.

Secondly, given $A$ in $\CC$, we have to find $T$ in $\TC$
and $L$ in $\FC$ such that we have a short exact sequence
$0 \to T \to A \to L \to 0$.

Since $(\TC_U,\FC_U)$ is a torsion theory on $\CC_U$, we have a short
exact sequence
\begin{equation}
\label{eq:ses tt U}
0
\longto (\p j^* A)_\tors
\longto \p j^* A
\longto (\p j^* A)_\free
\longto 0
\end{equation}

By adjunction, we have morphisms
\[
\p j_! (\p j^* A)_\tors
\elem{f} A
\elem{g} \p j_* (\p j^* A)_\free
\]
and the morphisms of \eqref{eq:ses tt U} are $\p j^* f$ and $\p j^* g$.
Let $S$ and $Q$ denote the images of $f$ and $g$. We have canonical
factorizations
\[
\xymatrix@!=.8cm{
\p j_! (\p j^* A)_\tors \ar[rr]^f \ar[rr] \ar@{->>}[dr]_{q_S}
&& A \ar[rr]^g \ar@{->>}[dr]_{q_Q}
&& \p j_* (\p j^* A)_\free
\\
& S \ar@{^(->}[ur]_{i_S}
&& Q \ar@{^(->}[ur]_{i_Q}
}
\]

By Lemma \ref{lemma:direct images rec tt}, since $(\p j^* A)_\tors$
is in $\TC_U$, the object $\p j_! (\p j^* A)_\tors$ is in $\TC$, so by
Lemma \ref{lemma:stab rec tt}, its quotient $S$ is also in $\TC$.
Similarly, $\p j_* (\p j^* A)_\free$ is in $\FC$ so its subobject $Q$
is also in $\FC$. By Lemma \ref{lemma:Hom(T,L) = 0}, it follows that
$\Hom_\CC(S,Q) = 0$. Thus $q_Q  i_S = 0$, and $i_S$ factors through
the kernel $b:B\injto A$ of $q_Q : A \onto Q$ as
$i_S = b \iota$, for some monomorphism $\iota : S\injto B$, and we can
identify $S$ with a subobject of $B$.
Now let $M = B/S$, and let $\pi : B \onto M$ be the canonical quotient
morphism.

The morphism $\p j^* f$ is a monomorphism, hence
$\p j^* q_S$, which is an epimorphism since $\p j^*$ is exact, is
actually an isomorphism. Similarly, $\p j^* i_Q$ is an isomorphism.
Thus $\p j^* b$ is the kernel of $\p j^* g$, and $\p j^* \iota$ is an
isomorphism as well. Applying $\p j^*$ to the short exact sequence
\[
0 \longto S \longto B \longto M \longto 0
\]
gives an exact sequence
\[
0 \longto \p j^* S \longto \p j^* B \longto \p j^* M \longto 0
\]
where the first morphism is an isomorphism, hence
$\p j^* M = 0$, and $M$ is in $\p i_* \CC_F$.
We have a short exact sequence
\[
0 \longto (\p i^* M)_\tors \longto \p i^* M \longto (\p i^* M)_\free
\longto 0
\]
Applying the exact functor $\p i_*$, we get a short exact sequence
\[
0 \longto \p i_* (\p i^* M)_\tors \longto M \longto \p i_* (\p i^* M)_\free
\longto 0
\]
and, by Lemma \ref{lemma:direct images rec tt},
$\p i_* (\p i^* M)_\tors$ is in $\TC$ and
$\p i_* (\p i^* M)_\free$ is in $\FC$.

Let $T$ denote the inverse image $\pi^{-1}(\p i_* (\p i^* M)_\tors)$
in $B$ (recall that $\pi : B \to M = B/S$ is the quotient morphism),
and let $L = A/T$.

We have a filtration $0 \subset S \subset T \subset B \subset A$ of
$A$, and the following short exact sequences:
\[
0 \longto S \longto T \longto \p i_* (\p i^* M)_\tors \longto 0
\]
which shows that $T$ is in $\TC$ by Lemma \ref{lemma:stab rec tt}, and
\[
0 \longto \p i_* (\p i^* M)_\free \longto L \longto Q \longto 0
\]
which shows that $L$ is in $\FC$ (by the same lemma), and
\[
0 \longto T \longto A \longto L \longto 0
\]
which completes the proof.
\end{proof}

Using these torsion theories on $\CC$, $\CC_F$ and $\CC_U$, one can
define, as in Subsection \ref{subsec:tt ts}, new $t$-structures on
$\DC$, $\DC_F$ and $\DC_U$, denoted with the
superscript $p_+$. Then the $t$-structure for $p_+$ on $\DC$ is obtained by
recollement from the $t$-structures for $p_+$ on $\DC_F$ and $\DC_U$.

Moreover, we have six interesting functors from $\CC_U \cap \CC_U^+$ to $\DC$
\begin{alignat}{5}
\p j_!
&\quad = \quad& \p\t^F_{\leqslant -2}\ j_* 
&\quad = \quad& \p\t^F_{\geqslant 0}\ j_!
\\
\pp j_!
&\quad = \quad& \p\t^F_{\leqslant -2_+}\ j_*
&\quad = \quad& \p\t^F_{\geqslant 0_+}\ j_!
\\ 
\p j_{!*}
&\quad = \quad& \p\t^F_{\leqslant -1}\ j_*
&\quad = \quad& \p\t^F_{\geqslant 1}\ j_!
\\ 
\pp j_{!*}
&\quad = \quad& \p\t^F_{\leqslant -1_+}\ j_*
&\quad = \quad& \p\t^F_{\geqslant 1_+}\ j_!
\\ 
\p j_*
&\quad = \quad& \p\t^F_{\leqslant 0}\ j_*
&\quad = \quad& \p\t^F_{\geqslant 2}\ j_!
\\ 
\pp j_*
&\quad = \quad& \p\t^F_{\leqslant 0_+}\ j_*
&\quad = \quad& \p\t^F_{\geqslant 2_+}\ j_!
\end{alignat}

The first of these functors has image in $\CC$, the last one in $\CC^+$,
and the other four in $\CC \cap \CC^+$.

We have a chain of morphisms
\[
\p j_! \longto \pp j_! \longto \p j_{!*} \longto \pp j_{!*}
\longto \p j_* \longto \pp j_*
\]
and distinguished triangles:
\begin{alignat}{6}
\label{tri:rec tt first}
\p j_! &\quad\longto\quad& \pp j_! &\quad\longto\quad& \p i_* \p H^{-1}_\tors i^* j_* [1] &\rtordu
\\
\pp j_! &\quad\longto\quad& \p j_{!*} &\quad\longto\quad& \p i_* \p H^{-1}_\free i^* j_* [1]&\rtordu
\\
\p j_{!*} &\quad\longto\quad& \pp j_{!*} &\quad\longto\quad& \p i_* \p H^0_\tors i^* j_* &\rtordu
\\
\pp j_{!*} &\quad\longto\quad& \p j_* &\quad\longto\quad& \p i_* \p H^0_\free i^* j_* &\rtordu
\\
\label{tri:rec tt last}
\p j_* &\quad\longto\quad& \pp j_* &\quad\longto\quad& \p i_* \p H^1_\tors i^* j_* [-1] &\rtordu
\end{alignat}
summarized by:
{\small
\[
\begin{array}{|*{6}{c|}}
\hline
\cdots
&{\SS \p i_* \p H^{-1}_\tors i^* j_*}
&{\SS \p i_* \p H^{-1}_\free i^* j_*}
&{\SS \p i_* \p H^0_\tors i^* j_*}
&{\SS \p i_* \p H^0_\free i^* j_*}
&{\SS \p i_* \p H^1_\tors i^* j_*}
\\
\hline
\multicolumn{1}{|r|}{\SS \p j_!}
&\multicolumn{1}{r|}{\SS \pp j_!}
&\multicolumn{1}{r|}{\SS \p j_{!*}}
&\multicolumn{1}{r|}{\SS \pp j_{!*}}
&\multicolumn{1}{r|}{\SS \p j_*}
&\multicolumn{1}{r|}{\SS \pp j_*}
\end{array}
\]
}

\subsection{Perverse $t$-structures}
\label{subsec:perv}

Let us go back to the setting of \ref{subsec:context}. We want to
define the $t$-structure defining the $\EM$-perverse sheaves on $X$
for the middle perversity $p$ (and, in case $\EM = \OM$, also for the
perversity $p_+$), following \cite{BBD}. Let us start with the case $\EM = \FM$. We will
consider pairs $(\XG,\LG)$ satisfying the following conditions:
\begin{assumption}
\label{ass:strata}
\mbox{}
\begin{enumerate}[(i)]
\item\label{XG}
$\XG$ is a partition of $X$ into finitely many locally closed smooth
  pieces, called strata, and the closure of a stratum is a union of strata.

\item\label{LG}
$\LG$ consists in the following data: for each stratum $S$ in $\XG$, 
a finite set $\LG(S)$ of isomorphism classes of irreducible locally constant
sheaves of $\FM$-modules over $S$.

\item\label{cons} For each $S$ in $\XG$ and for each $\FC$ in $\LC(S)$, if $j$
denotes the inclusion of $S$ into $X$, then the $R^nj_* \FC$ are
$(\XG,\LG)$-constructible, with the definition below.
\end{enumerate}
\end{assumption}

A sheaf of $\FM$-modules is $(\XG,\LG)$-constructible if and only if
its restriction to each stratum $S$ in $\XG$ is locally constant and
a finite iterated extension of irreducible locally constant sheaves
whose isomorphism class is in $\LG(S)$. 
We denote by $D^b_{\XG,\LG}(X,\FM)$ the full subcategory of $D^b(X,\FM)$
consisting of the $(\XG,\LG)$-constructible complexes, that is, whose
cohomology sheaves are $(\XG,\LG)$-constructible.

We say that $(\XG',\LG')$ refines $(\XG,\LG)$ if each stratum $S$ in $\XG$
is a union of strata in $\XG'$, and all the $\FC$ in $\LG(S)$ are
$(\XG',\LG')$-constructible, that is, $(\XG'_{|S},\LG_{\left|\XG'_{|S}\right.})$-constructible.

The condition (\ref{cons})
ensures that for $U \stackrel{j}{\injto} V \subset X$ locally closed
and unions of strata, the functors $j_*$, $j_!$ (resp. $j^*$, $j^!$) send
$D^b_{\XG,\LG}(U,\FM)$ into $D^b_{\XG,\LG}(V,\FM)$ (resp.
$D^b_{\XG,\LG}(V,\FM)$ into $D^b_{\XG,\LG}(U,\FM)$).
It follows from the constructibility theorem for $j_*$ (SGA 4$\frac{1}{2}$)
that any pair $(\XG',\LG')$ satisfying (\ref{XG}) and (\ref{LG}) can be refined
into a pair $(\XG,\LG)$ satisfying (\ref{XG}), (\ref{LG}) and
(\ref{cons}) (see \cite[\S 2.2.10]{BBD}).

So let us fix a pair $(\XG,\LG)$ as above. Then we define the full subcategories
$\p D_{\XG,\LG}^{\leqslant 0}(X,\FM)$ and $\p D_{\XG,\LG}^{\geqslant 0}(X,\FM)$
of $D^b_{\XG,\LG}(X,\FM)$ by
\begin{gather*}
A\in \p D^{\leqslant 0}_{\XG,\LG}(X,\EM) \Iff
\forall S \in \XG,\ \forall n > - \dim S,\ 
\HC^n i_S^* A = 0
\\
A\in \p D^{\geqslant 0}_{\XG,\LG}(X,\EM) \Iff
\forall S \in \XG,\ \forall n < - \dim S,\ 
\HC^n i_S^! A = 0
\end{gather*}
for any $A$ in $D^b_{\XG,\LG}(X,\FM)$, where $i_S$ is the inclusion of the stratum $S$.

One can show by induction on the number of strata that
this defines a $t$-structure on $D^b_{\XG,\LG}(X,\FM)$, by repeated
applications of Theorem \ref{th:recollement}. On a stratum, we
consider the natural $t$-structure shifted by $\dim S$, and we glue
these $t$-structures successively.

The $t$-structure on $D^b_{\XG',\LG'}(X,\FM)$ for a finer pair
$(\XG,\LG)$ induces the same $t$-structure on $D^b_{\XG,\LG}(X,\FM)$,
so passing to the limit we obtain a $t$-structure on $D^b_c(X,\FM)$,
which is described by the conditions \eqref{def:p <= 0} and
\eqref{def:p >= 0} of Subsection \ref{subsec:context}.

Over $\OM / \varpi^n$, we proceed similarly. An object $K$ of
$D^b_c(X,\OM/\varpi^n)$ is $(\XG,\LG)$-constructible if all the
$\varpi^i \HC^j K / \varpi^{i + 1} \HC^j K $ are $(\XG,\LG)$-constructible
as $\FM$-sheaves.

Over $\OM$, since our field $k$ is finite or algebraically closed, we
can use Deligne's definition of $D^b_c(X,\OM)$ as the projective
$2$-limit of the triangulated categories $D^b_c(X,\OM/\varpi^n)$. The
assumption insures that it is triangulated. We have triangulated functors
$\OM/\varpi^n \otimes^\LM_\OM (-) : D^b_c(X,\OM) \to D^b_c(X,\OM/\varpi^n)$,
and in particular $\FM \otimes^\LM_\OM (-)$. We will often omit from
the notation $\otimes^\LM_\OM$ and simply write $\FM(-)$. The functor
$\HC^i : D^b_c(X,\OM) \to \Sh(X,\OM)$ is defined by sending an object
$K$ to the projective system of the $\HC^i(\OM/\varpi^n \otimes^\LM_\OM K)$.
We have exact sequences:
{\small
\begin{equation*}
0 \longto \OM/\varpi^n \otimes_\OM \HC^i(K)
\longto \HC^i(\OM/\varpi^n \otimes^\LM_\OM K)
\longto \Tor_1^\OM(\OM/\varpi^n,\HC^{i + 1}(K))
\longto 0
\end{equation*}
}

Let $D^b_{\XG,\LG}(X,\OM)$ be the full subcategory of $D^b_c(X,\FM)$
consisting of the objects $K$ such that for some (or any) $n$,
$\OM/\varpi^n \otimes^\LM_\OM K$ is in $D^b_{\XG,\LG}(X,\OM/\varpi^n)$, or
equivalently, such that the $\F\otimes_\OM\HC^i K$ are $(\XG,\LG)$-constructible. 
We define the $t$-structure for the perversity $p$ on
$D^b_{\XG,\LG}(X,\OM)$ as above. Its heart is the abelian category
$\p \MC_{\XG,\LG}(X,\OM)$. Since it is $\OM$-linear, it is endowed
with a natural torsion theory, and we can define another
$t$-structure as in Subsection \ref{subsec:tt ts}, and we will say that it
is associated to the perversity $p_+$. By Subsection \ref{subsec:tt recollement},
it can also be obtained by recollement.
Passing to the limit, we get two $t$-structures on $D^b_c(X,\OM)$, for
the perversities $p$ and $p_+$, which can be characterized by the conditions
\eqref{def:p <= 0}, \eqref{def:p >= 0}, \eqref{def:p+ <= 0}
and \eqref{def:p+ >= 0} of Subsection \ref{subsec:context}.

An object $A$ of $D^b_c(X,\OM)$ is in $\p D^{\leqslant 0}(X,\OM)$
(resp. $\pp D^{\geqslant 0}(X,\OM)$)
if and only if $\FM A$ is in $\p D^{\leqslant 0}(X,\FM)$
(resp. $\pp D^{\geqslant 0}(X,\FM)$).

If $A$ is an object in $\p\MC(X,\OM)$, then
$\FM A$ is in $\p\MC(X,\FM)$ if and only if $A$ is
torsion-free (that is, if and only if $A$ is also $p_+$-perverse).
Then we have $\FM A = \Coker \varpi.1_A$
(the cokernel being taken in $\p\MC(X,\OM)$).

Similarly, if $A$ is an object in $\pp\MC(X,\OM)$, then
$\FM A$ is in $\p\MC(X,\FM)$ if and only if $A$ is
divisible (that is, if and only if $A$ is also $p$-perverse).
Then we have $\FM A = \Ker \varpi.1_A [1]$
(the kernel being taken in $\pp\MC(X,\OM)$).

To pass from $\OM$ to $\KM$, we simply apply $\KM \otimes_\OM
(-)$. Thus $D^b_c(X,\KM)$ is the category with the same objects as
$D^b_c(X,\OM)$, and morphisms
\[
\Hom_{D^b_c(X,\KM)}(A,B) = \KM \otimes_\OM \Hom_{D^b_c(X,\OM)}(A,B)
\]
We write $D^b_c(X,\KM) = \KM\otimes_\OM D^b_c(X,\OM)$. We also have
$\Sh(X,\KM) = \KM\otimes_\OM \Sh(X,\OM)$. Then we define the full
subcategory $D^b_{\XG,\LG}(X,\KM)$ of $D^b_c(X,\KM)$ as the image of 
$D^b_{\XG,\LG}(X,\OM)$. The $t$-structures $p$ and $p_+$ on
$D^b_{\XG,\LG}(X,\OM)$ give rise to a single $t$-structure $p$ on
$D^b_{\XG,\LG}(X,\KM)$, because torsion objects are killed by $\KM
\otimes_\OM (-)$. This perverse $t$-structure can be defined by
recollement. Passing to the limit, we get the perverse $t$-structure
on $D^b_c(X,\KM)$ defined by \eqref{def:p <= 0} and \eqref{def:p >= 0}. We
have $\p \MC(X,\KM) = \KM \otimes_\OM \p\MC(X,\OM)$.

\subsection{Modular reduction and truncation functors}\label{subsec:F recollement}

Modular reduction does not commute with truncation functors.
We will now study the failure of commutativity between these functors.
Recall that, to simplify the notation, we write $\FM (-)$ for
$\FM\otimes^\LM_\OM(-)$.

\begin{proposition}
\label{prop:Ft}
For $A \in D^b_c(X,\OM)$ and $n \in \ZM$, we have distinguished triangles:
\begin{alignat}{6}
\label{tri:Fi iF}
\F \t_{\leqslant n}\,A &\quad\longto\quad&
\t_{\leqslant n}\,\F A &\quad\longto\quad&
\HC^{-1}(\F \HC^{n+1}_\tors A) [-n] \rightsquigarrow
\\
\label{tri:iF Fi+}
\t_{\leqslant n}\,\F A &\quad\longto\quad&
\F  \t_{\leqslant n_+} A &\quad\longto\quad&
\HC^0 ( \F \HC^{n+1}_\tors A) [-n - 1] \rightsquigarrow
\\
\label{tri:Fi+ Fi+1}
\F  \t_{\leqslant n_+} A &\quad\longto\quad&
\F  \t_{\leqslant n+1}\,A &\quad\longto\quad&
\F \HC^{n+1}_\free A [-n-1] \rightsquigarrow
\end{alignat}

In particular,
\begin{alignat}{3}
\label{isom:Fi iF Fi+}
\HC^{n+1}_\tors A = 0 &\quad \Imp \quad&
\F \t_{\leqslant n}\,A \isom \t_{\leqslant n}\,\F A
\isom \F \t_{\leqslant n_+} A
\\
\label{isom:Fi+ Fi+1}
\HC^{n+1}_\free A = 0 &\quad \Imp \quad&
\F  \t_{\leqslant n_+} A \isom \F  \t_{\leqslant n+1}\,A
\qquad
\end{alignat}
\end{proposition}

\begin{proof}
We have a distinguished triangle \eqref{eq:tri free}
\[
\t_{\leqslant n_+} A \to \t_{\leqslant n + 1}\,A \to \HC^{n + 1}_\free A[-n - 1] \rtordu
\]
in $D^b_c(X,\OM)$. Applying $\FM(-)$, we get the
triangle (\ref{tri:Fi+ Fi+1}).
If $\HC^{n+1}_\free A$, this reduces to the isomorphism \eqref{isom:Fi+ Fi+1}.

We also have a distinguished triangle \eqref{eq:tri tors}
\[
\t_{\leqslant n}\,A \to \t_{\leqslant n_+} A \to \HC^{n + 1}_\tors A[-n - 1] \rtordu
\]
in $D^b_c(X,\OM)$. Applying $\FM(-)$, we get a distinguished triangle in
$D^b_c(X,\FM)$
\begin{equation}
\F \t_{\leqslant n}\,A \to \F \t_{\leqslant n_+} A \to
\F \HC^{n + 1}_\tors A[-n - 1] \rtordu
\end{equation}
On the other hand, we have a distinguished triangle
\begin{equation}
\Tor_1^\OM(\FM, \HC^{n + 1}_\tors A)[-n]
\to \F \HC^{n + 1}_\tors A[-n - 1]
\to \FM\otimes_\OM\HC^{n + 1}_\tors A[-n - 1]
\end{equation}

By the dual octahedron axiom of triangulated categories (the TR 4'
axiom, see \cite{BBD}), we have an octahedron diagram
{\small
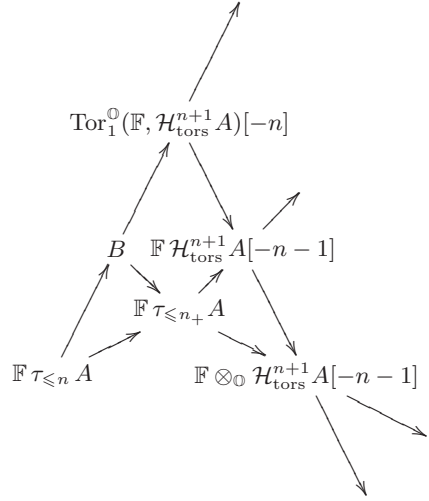
\begin{equation}
\label{octa:Fi iF Fi+}
\tag{$\Omega$}
\octa
{\F \t_{\leqslant n}\,A}
{B}
{\Tor_1^\OM(\FM, \HC^{n + 1}_\tors A)[-n]}
{\F \t_{\leqslant n_+} A}
{\F \HC^{n + 1}_\tors A[-n - 1]}
{\FM\otimes_\OM\HC^{n + 1}_\tors A[-n - 1]}
{@!=0cm}
\end{equation}
}
for some $B$ in $D^b_c(X,\OM)$.

The triangle
$(\F \t_{\leqslant n}\, A,\  B,\ \Tor_1^\OM(\FM, \HC^{n + 1}_\tors A)[-n])$
shows that
$B$ lies in $D^{\leqslant n}_c(X,\FM)$, and then
the triangle
$(B,\ \F \t_{\leqslant n_+} A,\ \FM\otimes_\OM\HC^{n + 1}_\tors A[-n - 1])$
shows that $B$ is (uniquely) isomorphic to
$\t_{\leqslant n}\, \F \t_{\leqslant n_+} A$.
Let us now show that
$\t_{\leqslant n}\, \F \t_{\leqslant n_+} A \simeq \t_{\leqslant n}\, \F A$.

By the TR 4 axiom \cite{BBD}, we have an octahedron diagram
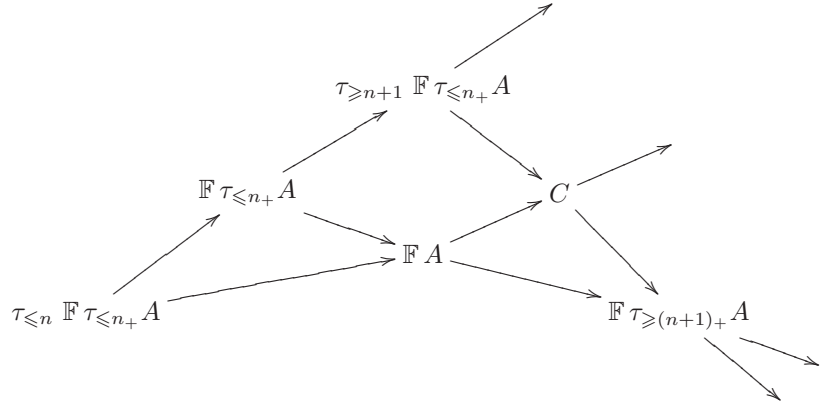
\[
\octa
{\t_{\leqslant n}\ \F \t_{\leqslant n_+} A}
{\F \t_{\leqslant n_+} A}
{\t_{\geqslant n + 1}\ \F \t_{\leqslant n_+} A}
{\F A}
{C}
{\F \t_{\geqslant (n + 1)_+} A}
{@R=.3cm @C=.3cm}
\]
for some $C$ in $D^b_c(X,\OM)$.

First, the triangle
$(\t_{\geqslant n + 1}\, \F \t_{\leqslant n_+} A,\ C,\ \F \t_{\geqslant (n + 1)_+} A)$
shows that $C$ lies in $D^{\geqslant n + 1}_c(X,\FM)$.
Secondly, the triangle
$(\t_{\leqslant n}\ \F \t_{\leqslant n_+} A,\ \F A,\ C)$
shows that
$B \simeq \t_{\leqslant n}\ \F \t_{\leqslant n_+} A \simeq \t_{\leqslant n}\ \F A$
and $C \simeq \t_{\geqslant n + 1}\ \F A$.

Hence the octahedron diagram \eqref{octa:Fi iF Fi+} contains the
triangles \eqref{tri:Fi iF} and \eqref{tri:iF Fi+}.
If $\HC^{n + 1}_\tors A = 0$, the diagram reduces to the isomorphisms
\eqref{isom:Fi iF Fi+}.
\end{proof}

We can summarize the Proposition by the following diagram:
{\small
\[
\begin{array}{*{6}{|r}|}
\hline
\multicolumn{3}{|c|}{\SS \HC^n \FM}
&\multicolumn{3}{c|}{\SS \HC^{n+1} \FM}
\\
\hline
{\SS \HC^0 \FM \HC^n_\tors}
&{\SS \FM \HC^n_\free}
&{\SS \HC^{-1} \FM \HC^{n+1}_\tors}
&{\SS \HC^0 \FM \HC^{n+1}_\tors}
&{\SS \FM \HC^{n+1}_\free}
&{\SS \HC^{-1} \FM \HC^{n+2}_\tors}
\\
\hline
{\SS \cdots}
&{\SS \FM \t_{\leq n}}
&{\SS \t_{\leq n} \FM}
&{\SS \FM \t_{\leq n_+}}
&{\SS \FM \t_{\leq n + 1}}
&{\SS \cdots}
\end{array}
\]
}
We have the same result if we replace $\t_{\leqslant n}$ by $\p
\t_{\leqslant n}$, and $\HC^n$ by $\p \HC^n$.

\subsection{Modular reduction and recollement}

Let us fix an open subvariety $j : U \to X$, with closed complement
$i : F\to X$. We want to see how the modular reduction behaves with
respect to this recollement situation.

For $A$ in $\p\MC(U,\OM) \cap \pp\MC(U,\OM)$, we have nine interesting
extensions of $\FM A$, out of which seven are automatically
perverse. These correspond to nine ways to truncate
$\FM j_* A = j_* \FM A$, three for each degree between $-2$
and $0$. Indeed, each degree is ``made of'' three parts:
the $\p\HC^0\,\FM(-)$ of the torsion part of the cohomology of $A$ of
the same degree, the reduction of the torsion-free part of the
cohomology of $A$ of the same degree, and the $\p\HC^{-1}\,\FM(-)$ of
the torsion part of the next degree (like a $\Tor_1$).

There is a variant of Proposition \ref{prop:Ft} for the functors
$\t^F_{\leqslant i}$ instead of $\t_{\leqslant i}$. We get the following diagram:
{\tiny
\[
\begin{array}{*{6}{|r}|}
\hline
\multicolumn{3}{|c|}{\SS i_* \HC^n i^* \FM}
&\multicolumn{3}{c|}{\SS i_* \HC^{n+1} i^* \FM}
\\
\hline
{\SS \HC^0 \FM i_* \HC^n_\tors i^*}
&{\SS \FM i_* \HC^n_\free i^*}
&{\SS \HC^{-1} \FM i_* \HC^{n+1}_\tors i^*}
&{\SS \HC^0 \FM i_* \HC^{n+1}_\tors i^*}
&{\SS \FM i_* \HC^{n+1}_\free i^*}
&{\SS \HC^{-1} \FM i_* \HC^{n+2}_\tors i^*}
\\
\hline
{\SS \cdots}
&{\SS \FM \t_{\leq i}}
&{\SS \t_{\leq i} \FM}
&{\SS \FM \t_{\leq i_+}}
&{\SS \FM \t_{\leq i + 1}}
&{\SS \cdots}
\end{array}
\]
}
The same remark applies if we use
$\p\t^F_{\leqslant n}$ instead of $\t^F_{\leqslant n}$.
Using Proposition \ref{prop:characterization},
we obtain a chain of morphisms:
\[
\FM \, \p j_!
\to \p j_!\, \FM
\to \FM \, \pp j_!
\to \FM \, \p j_{!*}
\to \p j_{!*}\, \FM
\to \FM \, \pp j_{!*}
\to \FM \, \p j_*
\to \p j_*\, \FM
\to \FM \, \pp j_*
\]
and distinguished triangles:
\begin{alignat}{6}
\label{tri:mod rec first}
\qquad
\FM \, \p j_! 
&\quad\longto\quad& \p j_!\, \FM
&\quad\longto\quad& \p\HC^{-1}\, \FM\, \p i_* \p\HC^{-1}_\tors\, i^* j_* [2]
\rightsquigarrow
\\ \T\qquad
\p j_!\, \FM
&\quad\longto\quad& \FM \, \pp j_!
&\quad\longto\quad& \p\HC^0\, \FM\, \p i_* \p\HC^{-1}_\tors\, i^* j_* [1]
\rightsquigarrow
\\ \T\qquad
\FM \, \pp j_!
&\quad\longto\quad& \FM \, \p j_{!*}
&\quad\longto\quad& \FM\, \p i_* \p\HC^{-1}_\free\, i^* j_* [1]
\rightsquigarrow
\\ \T\qquad
\FM \, \p j_{!*}
&\quad\longto\quad& \p j_{!*}\, \FM
&\quad\longto\quad& \p\HC^{-1}\, \FM\, \p i_* \p\HC^0_\tors\, i^* j_* [1]
\rightsquigarrow
\\ \T\qquad
\p j_{!*}\, \FM
&\quad\longto\quad& \FM \, \pp j_{!*}
&\quad\longto\quad& \p\HC^0\, \FM\, \p i_* \p\HC^0_\tors\, i^* j_*
\rightsquigarrow
\\ \T\qquad
\FM \, \pp j_{!*}
&\quad\longto\quad& \FM \, \p j_*
&\quad\longto\quad& \FM\, \p i_* \p\HC^0_\free\, i^* j_*
\rightsquigarrow
\\ \T\qquad
\FM \, \p j_*
&\quad\longto\quad& \p j_*\, \FM
&\quad\longto\quad& \p\HC^{-1}\, \FM\, \p i_* \p\HC^1_\tors\, i^* j_*
\rightsquigarrow
\\ \T\qquad
\label{tri:mod rec last}
\p j_*\, \FM  
&\quad\longto\quad& \FM \, \pp j_*
&\quad\longto\quad& \p\HC^0\, \FM\, \p i_* \p\HC^1_\tors\, i^* j_* [-1]
\rightsquigarrow
\end{alignat}

\medskip

In particular, for $A$ in $\p\MC(U,\OM) \cap \pp\MC(U,\OM)$, we have:
\begin{eqnarray}
\qquad
\p\HC^{-1}_\tors\, i^* j_* A = 0
\quad &\Imp& \quad 
\FM \, \p j_!\, A \elem{\sim} \p j_!\, \FM\, A \elem{\sim} \FM \, \pp j_!\, A
\\
\qquad
\p\HC^{-1}_\free\, i^* j_* A = 0
\quad &\Imp& \quad\qquad 
\FM \, \pp j_!\, A \elem{\sim} \FM \, \p j_{!*}\, A
\\
\qquad
\p\HC^0_\tors\, i^* j_* A = 0
\quad &\Imp& \quad 
\FM \, \p j_{!*}\, A \elem{\sim} \p j_{!*}\, \FM\, A \elem{\sim} \FM \, \pp j_{!*}\, A
\\
\qquad
\p\HC^0_\free\, i^* j_* A = 0
\quad &\Imp& \quad\qquad 
\FM \, \pp j_{!*}\, A \elem{\sim} \FM \, \p j_*\, A
\\
\qquad
\p\HC^1_\tors\, i^* j_* A = 0
\quad &\Imp& \quad 
\FM \, \p j_*\, A \elem{\sim} \p j_*\, \FM\, A \elem{\sim} \FM \, \pp j_*\, A
\end{eqnarray}

\subsection{Decomposition numbers}\label{subsec:decnum}

Let $X$ be endowed with a pair $(\XG,\LG)$ satisfying the conditions
(\ref{XG}), (\ref{LG}) and (\ref{cons}) of Section \ref{subsec:perv}.
Let $\PG$ be the set of pairs $(\OC,\LC)$ where $\OC \in \XG$ and $\LC \in \LG(\OC)$.
Let $K_0^{\XG,\LG}(X,\FM)$ be the Grothendieck group of the triangulated category
$D^b_{\XG,\LG}(X,\FM)$.

For $\OC \in \XG$, let $j_\OC : \OC \to X$ denote the inclusion.
For $(\OC,\LC) \in \PG$, let us denote by
\begin{eqnarray}
\0 \JC_!(\OC, \LC) &=& \0 {j_\OC}_!\ (\LC[\dim \OC])
\end{eqnarray}
the extension by zero of the local system $\LC$, shifted by $\dim \OC$.
We also introduce the following notation for the three perverse extensions.
\begin{eqnarray}
\p \JC_!(\OC, \LC) &=& \p {j_\OC}_!\ (\LC[\dim \OC])\\
\p \JC_{!*}(\OC, \LC) &=& \p {j_\OC}_{!*} (\LC[\dim \OC])\\
\p \JC_*(\OC, \LC) &=& \p {j_\OC}_*\ (\LC[\dim \OC])
\end{eqnarray}

We have
\begin{equation}
K_0^{\XG,\LG}(X,\FM) 
\simeq K_0(\Sh_{\XG,\LG}(X,\FM)) 
\simeq K_0(\p\MC_{\XG,\LG}(X,\FM))
\end{equation}
If $K \in D^b_{\XG,\LG}(X,\FM)$, then we have
\[
[K] = \sum_{i \in \ZM} (-1)^i [\HC^i(K)] = \sum_{j \in \ZM} (-1)^j [\p\HC^j(K)]
\]
in $K_0^{\XG,\LG}(X,\FM)$.

This Grothendieck group is free over $\ZM$, and admits the following bases
\begin{eqnarray*}
\BC_0 &=& (\0\JC_!(\OC,\LC))_{(\OC,\LC)\in\PG}\\
\BC_! &=& (\p\JC_!(\OC,\LC))_{(\OC,\LC)\in\PG}\\
\BC_{!*} &=& (\p\JC_{!*}(\OC,\LC))_{(\OC,\LC)\in\PG}\\
\BC_* &=& (\p\JC_*(\OC,\LC))_{(\OC,\LC)\in\PG}
\end{eqnarray*}

For $C \in K_0^{\XG,\LG}(X,\FM)$, let us define the integers
$\chi_{(\OC,\LC)}(C)$, for $(\OC,\LC)\in\PG$, by
the relations
\begin{eqnarray*}
C &=& \sum_{(\OC,\LC)\in\PG} \chi_{(\OC,\LC)}(C)\ [\0\JC_!(\OC,\LC)]
\end{eqnarray*}

For $? \in \{!,!*,*\}$, the complex $\p\JC_?(\OC,\LC)$ extends
the shifted local system $\LC[\dim \OC]$, and is supported on $\ov\OC$. This implies
\begin{equation}
\chi_{(\OC',\LC')}(\p\JC_?(\OC,\LC)) = 0
\text{ unless } \ov\OC' \subsetneq \ov\OC \text{ or } (\OC',\LC') = (\OC,\LC)
\end{equation}
and
\begin{equation}
\chi_{(\OC,\LC)}(\p\JC_?(\OC,\LC)) = 1
\end{equation}

In other words, the three bases $\BC_!$, $\BC_{!*}$ and $\BC_*$
are unitriangular with respect to the basis $\BC_0$.
This implies that they are also unitriangular with respect to each other.
In fact, we already knew it by Proposition \ref{prop:top socle},
since $\p\JC_!(\OC,\LC)$ (resp. $\p\JC_*(\OC,\LC)$)
has a top (resp. socle) isomorphic to $\p\JC_{!*}(\OC,\LC)$,
and the radical (resp. the quotient by the socle) is supported on $\ov\OC\setminus\OC$.
In particular, for $?\in\{!,*\}$, we have
\begin{equation}
[\p\JC_?(\OC,\LC) : \p\JC_{!*}(\OC',\LC')] = 0
\text{ unless } \ov\OC' \subsetneq \ov\OC \text{ or } (\OC',\LC') = (\OC,\LC)
\end{equation}
and
\begin{equation}
[\p\JC_?(\OC,\LC) : \p\JC_{!*}(\OC,\LC)] = 1
\end{equation}

Let $K_0^{\XG,\LG}(X,\KM)$ be the Grothendieck group of the triangulated category
$D^b_{\XG,\LG}(X,\KM)$. As for the case $\EM = \FM$, it can be identified with
the Grothendieck groups of 
$\Sh_{\XG,\LG}(X,\KM)$ and $\p\MC_{\XG,\LG}(X,\KM)$.

Now, let $K$ be an object of $D^b_{\XG,\LC}(X,\KM)$.  If $K_\OM$ is an
object of $D^b_{\XG,\LC}(X,\OM)$ such that $\KM \otimes_\OM K_\OM
\simeq K$, we can consider $[\FM K_\OM]$ in
$K_0^{\XG,\LC}(X,\FM)$. This class does not depend on the choice of
$K_\OM$ (note that the modular reduction of a torsion object has a
zero class in the Grothendieck group: if we assume, for simplicity,
that we have only finite monodromy, then by dévissage we can reduce to
the analogue result for finite groups). In fact, it depends only on
the class $[K]$ of $K$ in $K_0^{\XG,\LC}(X,\KM)$. So we have a
well-defined morphism
\begin{equation}
d : K_0^{\XG,\LC}(X,\KM) \longto K_0^{\XG,\LC}(X,\FM)
\end{equation}

For $(\OC,\LC)
\in \PG$, we can consider the decomposition number $[\FM K_\OM :
\p\JC_{!*}(\OC,\LC)]$, where $K_\OM$ is any object of
$D^b_{\XG,\LC}(X,\OM)$ such that $\KM K_\OM \simeq K$.

\subsection{Equivariance}
\label{subsec:orbits}

We now introduce $G$-equivariant perverse sheaves in the sense of
\cite[\S 0]{ICC}, \cite[\S 4.2]{LET}.

Let $G$ be a \emph{connected} algebraic group acting on a variety $X$.
Let $\rho : G \times X \to X$ be the morphism defining the action, and let
$p : G \times X \to X$ be the second projection. A sheaf
$F$ on $X$ is $G$-\emph{equivariant} if there is an isomorphism
$\a : p^* F \isom \rho^* F$. In that case, we can choose $\a$ in a unique way
such that the induced isomorphism $i^*(\a) : F \to F$ is the identity,
where $i : X \to G \times X$ is defined by $i(x) = (1_G, x)$.

If $f : X \to Y$ is a $G$-equivariant morphism, the functors $\0 f^*$,
$\0 f_*$ and $\0 f_!$ take $G$-equivariant sheaves to $G$-equivariant sheaves.

Let $\Sh_G(X,\EM)$ be the category whose objects are the $G$-equivariant $\EM$-sheaves
on $X$, and such that the morphisms between two objects $F_1$ and $F_2$
are the morphisms $\phi$ in $\Sh(X,\EM)$ such that the following diagram commutes
\[
\xymatrix{
p^* F_1 \ar[r]^{p^*\phi} \ar[d]_{\a_1} &
p^* F_2 \ar[d]^{\a_2}\\
\rho^* F_1 \ar[r]_{\rho^*\phi} &
\rho^* F_2
}
\]
where $\a_j$ is the unique isomorphism such that $i^*(\a_j)$ is the
identity for $j = 1,2$.
Then it turns out that $\Sh_G(X,\EM)$ is actually a full subcategory of
$\Sh(X,\EM)$.

For a general complex in $D^b_c(X,\EM)$, the notion of $G$-equivariance is more
delicate. However, for a perverse sheaf we can take the same definition as above,
and again the isomorphism $\a$ can be normalized with the same condition.
If $f$ is a $G$-equivariant morphism, then the functors $\p\HC^j f^*$,
$\p\HC^j f^!$, $\p\HC^j f_*$ and $\p\HC^j f^!$ take $G$-equivariant perverse
sheaves to $G$-equivariant perverse sheaves.

We define in the same way the category $\p\MC_G(X,\EM)$ of
$G$-equivariant perverse $\EM$-sheaves,
and again it is a full subcategory of $\p\MC(X,\EM)$. Moreover, it is stable
by subquotients. The simple objects in $\p\MC_G(X,\EM)$ are the intermediate
extensions of irreducible $G$-equivariant $\EM$-local systems on $G$-stable locally closed
smooth irreducible subvarieties of $X$.

Suppose $\EM$ is a field.
If $\OC$ is a homogeneous space for $G$, let $x$ be a point in $\OC$,
and let $A_G(x) = C_G(x)/C_G^0(x)$. Then the set of isomorphism
classes of irreducible $G$-equivariant $\EM$-local systems is in
bijection with the set $\Irr \EM A_G(x)$ of isomorphism classes of
irreducible representations of the group algebra $\EM A_G(x)$.

Suppose $X$ is a $G$-variety with finitely many orbits.  Then we can
take the stratification $\XG$ of $X$ by its $G$-orbits.  The orbits
are indeed locally closed by \cite[Lemma 2.3.1]{SPR}, and they are
smooth. For each $G$-orbit $\OC$ in $X$, let $x_\OC$ be a closed point
in $\OC$.  For $\LG(\OC)$ we take all the irreducible
$G$-equivariant $\FM$-local systems, so that we can identify
$\LG(\OC)$ with $\Irr \FM A_G(x_\OC)$.

Suppose $\EM$ is a field. Let $K_0^G(X,\EM)$ be the Grothendieck group
of the triangulated category $D^b_{\XG,\LG}(X,\EM)$. Then we have
\begin{equation}
K_0^G(X,\EM) = K_0(\p\MC_G(X,\EM)) = K_0(\Sh_G(X,\EM)) \simeq 
\bigoplus_{\OC} K_0(\Irr \EM A_G(x_\OC))
\end{equation}
If $K \in D^b_{\XG,\LG}(X,\EM)$, then we have
\[
[K] = \sum_{i \in \ZM} (-1)^i [\HC^i(K)] = \sum_{j \in \ZM} (-1)^j [\p\HC^j(K)]
\]
in $K_0^G(X,\EM)$.

Let $\PG_\EM$ be the set of pairs $(\OC,\LC)$ with $\OC \in \XG$ and
$\LC$ an irreducible $G$-equivariant $\EM$-local system on $\OC$
(corresponding to an irreducible representation $L$ of $\EM
A_G(x_\OC)$). Then we have bases
$\BC_0^\EM = (\0 j_!(\OC,\LC))_{(\OC,\LC)\in\PG_\EM}$,
$\BC_!^\EM = (\p j_!(\OC,\LC))_{(\OC,\LC)\in\PG_\EM}$,
$\BC_{!*}^\EM = (\p j_{!*}(\OC,\LC))_{(\OC,\LC)\in\PG_\EM}$,
$\BC_*^\EM = (\p j_*(\OC,\LC))_{(\OC,\LC)\in\PG_\EM}$.
Note that, if $\ell$ does not divide the $|A_G(x_\OC)|$, then we can
identify $\PG_\KM$ with $\PG_\FM$.

The transition matrices from $\BC_0^\EM$ to $\BC_?^\EM$ (for $?\in\{!,!*,*\}$)
are unitriangular, and also the transition matrices from
$\BC_{!*}^\EM$ to $\BC_?^\EM$ (for $?\in\{!,*\}$).

As in the last section, we have a morphism
\[
d : K_0^G(X,\KM) \longto K_0^G(X,\FM)
\]

The matrix of $d$ with respect to the bases $\BC_0^\EM$ is just
a product of blocks indexed by the orbits $\OC$, the block
corresponding to $\OC$ being the decomposition matrix of the finite
group $A_G(x_\OC)$. If $\ell$ does not divide the $|A_G(x_\OC)|$, this
is just the identity matrix.

We are interested in the matrix of $d$ in the bases 
$\BC_{!*}^\EM$. That is, we want to study the decomposition numbers
\[
d^X_{(\OC,\LC),(\OC',\LC')}
= [\FM\JC_{!*}(\OC,\LC_\OM) : \JC_{!*}(\OC',\LC')]
\]
for $(\OC,\LC) \in \PG_\KM$ and $(\OC',\LC') \in \PG_\FM$,
where $\LC_\OM$ is an integral form for $\LC$.
Recall that, if $\ell$ does not divide the $|A_G(x)|$, then we can
identify $\PG_\KM$ with $\PG_\FM$.


\section{Some techniques}
\label{sec:techniques}

By the results in Subsection \ref{subsec:orbits}, to compute
decomposition numbers in a $G$-equivariant setting, it is enough to
compute the stalks of the intersection cohomology complexes over $\KM$
and $\FM$, with the actions of the groups $A_G(x)$
(then we just have to solve a triangular linear system). In the
applications, these are usually known over $\KM$ but not over $\FM$.
It is harder to compute over $\FM$: for example, one cannot use
arguments involving counting points, or the Decomposition Theorem.
We are going to see some methods that can be used in the modular case.
Some of them will be illustrated in the next sections. The results
about $\EM$-smoothness will be illustrated in \cite{modspringer} (see
\cite{these}), in relation with the special pieces of the nilpotent cone.

\subsection{Semi-small morphisms}
\label{subsec:semismall}

The classical results about semi-small and small projective morphisms
still apply in the modular case. Nevertheless, unless we have a small resolution,
they are less useful to
determine the stalks of the intersection cohomology complexes, because
the Decomposition Theorem \cite{BBD} does not hold in this case.

\begin{definition}
A morphism $\pi : \Xti \to X$ is \emph{semi-small} is there is a stratification
$\XG$ of $X$ such that the for all strata $S$ in $\XG$, and for all closed
points $s$ in $S$, we have $\dim \pi^{-1}(s) \leqslant \frac{1}{2}\codim_X(S)$.
If moreover these inequalities are strict for all strata of positive codimension,
we say that $\pi$ is \emph{small}.
\end{definition}

Recall that $\Loc(S,\EM)$ is the full subcategory of $\Sh(X,\EM)$
consisting of the $\EM$-local systems. It is the heart of the
$t$-category $D^b_{\Loc}(S,\EM)$ which is the full subcategory of
$D^b_c(S,\EM)$ of objects $A$ such that all the $\HC^i A$ are local
systems, with the $t$-structure induced by the natural $t$-structure
on $D^b_c(S,\EM)$. For $\EM = \OM$, according to the definition given
after Proposition \ref{prop:tt ts}, we have an abelian category
$\Loc^+(S,\OM)$, which is the full subcategory of $D^b_c(S,\OM)$
consisting of the objects $A$ such that $\HC^0 A$ is a torsion-free
$\OM$-local system, and $\HC^1 A$ is a torsion $\OM$-local system.

\begin{proposition}\label{prop:small}
Let $\pi : \Xti \to X$ be a surjective, proper and separable morphism,
with $\Xti$ smooth irreducible of dimension $d$.
Let $\LC$ be in $\Loc(\Xti,\EM)$. Let us consider the complex $K = \pi_!\ \LC[d]$.
\begin{enumerate}[(i)]
\item If $\pi$ is semi-small, then $\dim X = d$ and $K$ is $p$-perverse.
\item If $\pi$ is small, then $K = \p j_{!*}\p j^* K$
for any inclusion $j : U \to X$ of a smooth open dense subvariety over which
$\pi$ is étale.
\end{enumerate}

In the case $\EM = \OM$, we can take $\LC$ in $\Loc^+(X,\OM)$ and
replace $p$ by $p_+$.
\end{proposition}

In the case $\EM = \KM$, the Decomposition Theorem \cite{BBD} says that
$K$ is the direct sum of its shifted perverse cohomology sheaves
and that each $\p H^i K$ is a semi-simple perverse sheaf.
If $\pi$ is semi-small, then only $\p H^0 K$ can be non-zero.
So, in the characteristic zero case, if $\pi$ is semi-small, the
intersection cohomology complex will be a direct summand
of the direct image of the constant perverse sheaf, the other simple
summands having strictly smaller support. These simple summands
correspond to the relevant pairs \cite{BM,BM2}. If $\pi$ is small,
then the only relevant stratum is the open stratum.

In the favorable case where we have a small resolution, to compute
the intersection cohomology stalks over any $\EM$, we are reduced to
compute the stalks of the direct image of the constant sheaf, that is,
the cohomology with $\EM$ coefficients of the fibers.

However, in the case of a semi-small resolution, the situation is less
favorable in characteristic $\ell$
than in characteristic zero. We can only say that the
intersection cohomology complex of $X$ is a subquotient of $K$. For
example, it can have non-zero stalks in odd degree, even if $K$ has
non-zero stalks only in even degree.

Now let us say what can happen when $\pi$ is a semi-small
morphism which is not a resolution. 
Since it is assumed to be
separable, there is an smooth open dense subvariety $j:X_0\injto X$ over which the
pullback $\pi_0 : \tilde X_0 \to X_0$ is finite étale. We can find a Galois
finite étale covering $Y$ of $X_0$, with Galois group $G$,
such that $\pi_0 : \tilde X_0 \to X$ is the subcovering corresponding
to a subgroup $H$ of $G$. Then the direct image under $\pi_0$ of the constant
perverse sheaf on $\tilde X_0$, which is just $\p j^* K$, is the local system corresponding to
the permutation representation $\EM [G/H]$ of $\EM G$. If $\ell$ does not divide
the index $|G:H|$, then the trivial module $\EM$ is a direct summand
of $\EM [G/H]$, and $\p j_{!*} \un\EM$ is a direct summand of
$\p j_{!*} \p j^* K$. Otherwise, $\EM$ is both a submodule and a
quotient of $\EM [G/H]$, so $\p j_{!*} \p j^* K$ will have
$\p j_{!*} \un\EM$ both as a suboject and as a quotient, but, besides
the other composition factors coming from $X_0$, there can be new
composition factors coming from the closed complement $F$ (thus
illustrating the non-exactness of $\p j_{!*}$).
If $\pi$ is small, then we have $K = \p j_{!*} \p j^* K$, but otherwise
$K$ can have composition factors coming from $F$ as subobjects and as
quotients, and $K = \p j_{!*} \p j^* K$ is just a subquotient of $K$.

\subsection{$\EM$-smoothness}
\label{subsec:E smooth}

We say that $X$ is $\EM$-smooth if $\ic_\EM(X)$ is reduced to
$\un \EM_X$. When $\EM = \OM$, we require this condition for both
perversities, $p$ and $p_+$. This property ensures that $X$ satisfies
Poincaré duality with $\EM$-coefficients. The notion of rational
smoothness was introduced by Deligne in \cite{DELIGNE}.

A smooth variety is $\EM$-smooth in all cases. If $X$ is $\KM$-smooth,
then it is $\FM_\ell$-smooth for all but finitely many $\ell$. If $X$
is not $\KM$-smooth, then it is $\FM_\ell$-smooth for no $\ell$.
On the other hand, $X$
is $\OM$-smooth if and only if it is $\FM_\ell$-smooth for all $\ell$. The next
proposition provides examples of varieties that are $\FM_\ell$-smooth
for some but not all primes $\ell$, such as the simple singularities
(in particular, they are not smooth).

\begin{proposition}
Let $H$ be a finite group of order prime to $\ell$.
If $X$ is an $\FM_\ell$-smooth $H$-variety, then $X/H$ is also
$\FM_\ell$-smooth.
\end{proposition}

Looking at the stalks, we can deduce the following information about
decomposition numbers: if a locally closed irreducible union of strata is
$\FM_\ell$-smooth, then the decomposition numbers involving the
intermediate extension of the constant perverse sheaf on the open
stratum and a simple perverse sheaf associated to any irreducible
modular local system on a smaller stratum in this union are all zero.

\subsection{Deligne's construction}
\label{subsec:deligne}

Initially, intersection homology was defined topologically, using
chains satisfying certain conditions with respect to the
stratification. This construction was sheafified:
intersection cohomology can be computed as the
hypercohomology of a complex, the intersection cohomology complex.
De\-ligne found a purely algebraic construction of the intersection
cohomology complex, making sense also when the base field has field positive
characteristic $p$ (in the étale topology). Then this was included in the
theory of perverse sheaves \cite{BBD}. The abstract setting is that of
a recollement situation. The intersection cohomology
complexes of irreducible closed subvarieties $Y$,
with coefficients in any irreducible local system on a smooth open
dense subvariety of $Y$, are the simple perverse sheaves (if the
stratification is fixed, one takes for $Y$ the closure of a stratum).
These intersection cohomology complexes coincide with the intermediate
extensions of the (shifted) local systems.
This works both with $\KM$-sheaves and $\FM$-sheaves. 

In the examples we will compute over $\FM$, Deligne's construction
will be the main tool, because most other approaches fail (we do not
have weights nor the Decomposition Theorem). So let us recall the
procedure to calculate these intermediate extensions.

Assume we have a pair $(\XG,\LG)$ as in Assumption \ref{ass:strata}.
Let $U_k$ be the union of the strata of dimension at least $- k$
(it is an open subvariety of $X$).
Let $j_k : U_{k-1} \injto U_k$ denote the open inclusion.
We have
\[
U_{-d} \subset \cdots \subset U_{-1} \subset U_0 = X
\]

\begin{proposition}
Let $A$ be a $p$-perverse $\EM$-sheaf on $U_k$. Let $j$ denote the
inclusion of $U_k$ into $X$. Then we have
\[
\p j_{!*} A = \t_{\leq -1} j_{0*} \cdots \t_{\leq k} j_{k+1*} A
\]

If $\EM = \OM$, we also have a similar formula with $p$ replaced by $p_+$
and $\t_{\leq i}$ replaced by $\t_{\leq i_+}$.
\end{proposition}

The proof uses the transitivity of $\p j_{!*}$, \eqref{eq:tFtU} and
Proposition \ref{prop:characterization}.
See \cite{BBD}, Proposition 2.1.11, Proposition 2.2.4 and 3.3.4.

Actually, in the examples we will compute, there will be only one step
(to go from one stratum to the union of two strata), so what we will
really use here is Proposition \ref{prop:characterization}.

\subsection{Cones}
\label{subsec:cones}

Let $Y \subset \PM^{N - 1}$ be a smooth projective variety of dimension $d - 1$.
We denote by $\pi : \AM^N \setminus \{0\} \to \PM^{N - 1}$
the canonical projection. Let $U = \pi^{-1}(Y) \subset \AM^N \setminus \{0\}$
and $X = \ov U = U \cup \{0\} \subset \AM^N$. They have dimension $d$.

We have a smooth open immersion $j : U \injto X$ and a closed immersion
$i : \{0\} \injto X$. If $d > 1$, then $j$ is not affine.

\begin{proposition}\label{prop:cone}
With the preceding notations, we have
\[
i^* j_* \EM \simeq \rg(U, \EM)
\]
\end{proposition}

Truncating appropriately, one deduces the fiber at $0$ of the complexes
$\p j_?\ \EM[d]$, where $? \in \{!,!*,*\}$, and similarly for $p_+$
if $\EM = \OM$.

More generally, we have the following result, which is contained in
\cite[Lemma 4.5 (a)]{KL2}. As indicated there,
in the complex case, this follows easily from topological considerations.

\begin{proposition}\label{prop:gencone}
Let $X$ be an irreducible closed subvariety of $\AM^N$ stable under the
$\GM_m$-action defined by
$\l(z_1,\ldots,z_N) = (\l^{a_1}z_1,\ldots,\l^{a_N}z_N)$, where
$a_1 > 0$, \ldots, $a_N > 0$.
Let $j : U = X \setminus \{0\} \to X$ be the open immersion, and
$i : \{0\} \to X$ the closed immersion. Then we have
\[
i^* j_* \EM \simeq \rg(U, \EM)
\]
\end{proposition}

So, if $U$ is smooth, the calculation of the intersection cohomology
complex stalks for $X$ is reduced to the calculation of the cohomology of $U$.

\subsection{Equivalent singularities}

\begin{definition}\label{def:equiv}
Given $X$ and $Y$ two varieties, and two points $x\in X$ and $y\in Y$,
we say that the singularity of $X$ at $x$ and the singularity of $Y$
at $y$ are smoothly equivalent, and we write $\Sing(X,x) =
\Sing(Y,y)$, if there exist a variety $Z$, a point $z\in Z$, and two
maps $\varphi : Z \to X$ and $\psi : Z \to Y$, smooth at $z$, with
$\varphi(z) = x$ and $\psi(z) = y$.

If an algebraic group $G$ acts on $X$, then $\Sing(X,x)$
depends only on the orbit $\OC$ of $x$. In that case, we write
$\Sing(X,\OC) := \Sing(X,x)$.
\end{definition}

In fact, there is an open subset $U$ of $Z$ containing $z$ where
$\varphi$ and $\psi$ are smooth, so after replacing $Z$ by $U$,
we can assume that $\varphi$ and $\psi$ are smooth on $Z$.

We have the following result (it follows from the remarks after Lemma
4.2.6.1. in \cite{BBD}).
\begin{proposition}
\label{prop:equiv}
Suppose that $\Sing(X,x) = \Sing(Y,y)$. Then the $\EM$-modules
$\ic(X,\EM)_x$ and $\ic(Y,\EM)_y$ are isomorphic.
\end{proposition}

\begin{remark}\label{rem:equiv}
Suppose we have 
a stratification $\XG$ of $X$ adapted to $\ic(X,\EM)$
and a stratification $\YG$ of Y adapted to $\ic(Y,\EM)$,
and let $\OC(x)$ and $\OC(y)$ denote the respective strata of $x$ and $y$.
Suppose we know $\ic(X,\EM)_x$ as an $\EM$-module with continuous
action of $\pi_1(\OC(x),x)$. The proposition then gives us
$\ic(Y,\EM)_y$ as an $\EM$-module, but it does not give the action of
$\pi_1(\OC(y),y)$. To determine the latter structure,
one needs more information.
\end{remark}

\section{Simple singularities}
\label{sec:simple}

In this section, we will calculate the intersection cohomology
complexes over $\KM$, $\OM$ and $\FM$ for rational double points, and
the corresponding decomposition numbers. We will also consider the
case of simple singularities of inhomogeneous type, that is,
rational double points with an associated group of symmetries.
It is necessary to keep track of the action of this finite group for
the final application, which is the calculation of the decomposition
numbers for equivariant perverse sheaves on the nilpotent cone of a
simple Lie algebra, involving the regular orbit and the subregular
orbit.

For the convenience of the reader, we will recall the main points in
the theory of simple singularities, following \cite{SLO2}, to which we
refer for more details. The application to the nilpotent cone uses the
result of Brieskorn and Slodowy \cite{BRI,SLO1,SLO2}, showing that the
singularity of the nilpotent cone along the subregular class is a
simple singularity of the corresponding type.

\subsection{Rational double points}
\label{subsec:double points}

We assume that $k$ is algebraically closed. Let $(X,x)$ be the
spectrum of a two-dimensional normal local $k$-algebra, where $x$ denotes
the closed point of $X$. Then $(X,x)$ is \emph{rational} if there is a
resolution $\pi : \wt X \to X$ of the singularities of $X$ such that
the higher direct images of the structural sheaf of $\wt X$ vanish,
that is, $R^q \pi_*(\OC_{\wt X}) = 0$ for $q > 0$. In fact, this
property is independent of the choice of a resolution. The rationality
property is stronger than the Cohen-Macaulay property.

If $\pi : \wt X \to X$ is a resolution, then the reduced exceptional
divisor $E = \pi^{-1}(x)_\text{red}$ is a finite union of irreducible
curves (in particular, $\pi$ is semi-small).
Since $X$ is a surface, there is a minimal resolution, unique
up to isomorphism, through which all other resolutions must factor.
For the minimal resolution of a simple singularity, these
curves will have a very special configuration.

Let $\G$ be an irreducible homogeneous Dynkin diagram, with set of
vertices $\D$.  We recall that a Dynkin diagram is \emph{homogeneous},
or \emph{simply-laced}, when the corresponding root system has only
roots of the same length. Thus $\G$ is of type $A_n$ ($n \geqslant
1$), $D_n$ ($n \geqslant 4$), $E_6$, $E_7$ or $E_8$.  The Cartan
matrix $C = (n_{\a,\b})_{\a,\b\in\D}$ of $\G$ satisfies $n_{\a,\a} =
2$ for all $\a$ in $\D$, and $n_{\a,\b} \in \{0,-1\}$ for all $\a \neq
\b$ in $\D$.

A resolution $\pi : \wt X \to X$ of the surface $X$, as above, has an
\emph{exceptional configuration of type $\G$} if all the irreducible
components of the exceptional divisor $E$ are projective lines, and if
there is a bijection $\a \mapsto E_\a$ from $\D$ to the set $\Irr(E)$
of these components such that the intersection numbers
$E_\a \cdot E_\b$ are given by the opposite of the Cartan matrix $C$,
that is, $E_\a \cdot E_\b = - n_{\a,\b}$ for $\a$ and $\b$ in $\D$. 
Thus we have a union of projective lines whose normal bundles in $\wt
X$ are isomorphic to the cotangent bundle $T^*\PM^1$, and two of them
intersect transversely in at most one point.

The minimal resolution is characterized by the fact that it has no
exceptional curves with self-intersection $-1$. Therefore, if the
resolution $\pi$ of the surface $X$ has an exceptional configuration
of type $\G$, then it is minimal.

\begin{theorem}
The following properties of a normal surface $(X,x)$ are equivalent.
\begin{enumerate}[(i)]
\item $(X,x)$ is rational of embedding dimension $3$ at $x$.
\item $(X,x)$ is rational of multiplicity $2$ at $x$.
\item $(X,x)$ is of multiplicity $2$ at $x$ and it can be resolved by
  successive blowing up of points.
\item The minimal resolution of $(X,x)$ has the exceptional
  configuration of an irreducible homogeneous Dynkin diagram.
\end{enumerate}
\end{theorem}

\begin{definition}
If any (hence all) of the properties of the preceding theorem is
satisfied, then $(X,x)$ is called a \emph{rational double point} or a
\emph{simple singularity}.
\end{definition}

\begin{theorem}
Let the characteristic of $k$ be good for the irreducible homogeneous
Dynkin diagram $\G$. Then there is exactly one rational double point
of type $\G$ up to isomorphism of Henselizations. Representatives of
the individual classes are given by the local varieties at $0 \in
\AM^3$ defined by the equations in the table below.

In each case, this equation is the unique relation (syzygy) between
three suitably chosen generators $X$, $Y$, $Z$ of the algebra
$k[\AM^2]^H$ of the invariant polynomials of $\AM^2$ under the action
of a finite subgroup $H$ of $SL_2$, given in the same table.

\[
\begin{array}{llcrc}
H && |H| & \text{equation of } \AM^2/H \subset \AM^3 & \G
\\
\hline
\CG_{n+1} & \text{cyclic} & n + 1 & X^{n + 1} + YZ = 0 & A_n
\\
\DG_{4(n-2)} & \text{dihedral} & 4(n - 2) & X^{n - 1} + X Y^2 + Z^2 = 0 & D_n
\\
\TG & \text{binary tetrahedral} & 24 & X^4 + Y^3 + Z^2 = 0 & E_6
\\
\OG & \text{binary octahedral} & 48 & X^3 Y + Y^3 + Z^2 = 0 & E_7 
\\
\IG & \text{binary icosahedral} & 120 & X^5 + Y^3 + Z^2 = 0 & E_8
\end{array}
\]

Moreover, if $k$ is of characteristic $0$, these groups are, up to
conjugation, the only finite subgroups of $SL_2$.
\end{theorem}

Thus, in good characteristic, every rational double point is, after
Henseli\-zation at the singular point, isomorphic to the corresponding
quotient $\AM^2/H$. When $p$ divides $n + 1$ (resp. $4(n - 2)$), the
group $\CG_{n+1}$ (resp. $\DG_{4(n-2)}$) is not reduced. We have the
following exact sequences
\begin{gather}
1 \longto \DG_8 \longto \TG \longto \CG_3 \longto 1\\
1 \longto \TG \longto \OG \longto \CG_2 \longto 1\\
1 \longto \DG_8 \longto \OG \longto \SG_3 \longto 1
\end{gather}
when the characteristic of $k$ is good for the Dynkin diagram attached
to each of the groups involved.

\subsection{Symmetries on rational double points}
\label{subsec:symmetries}

To each inhomogeneous irreducible Dynkin diagram $\G$ we associate a
homogeneous diagram $\wh \G$ and a group $A(\G)$ of automorphisms of $\wh
\G$, as follows.

\[
\begin{array}{|c|c|c|c|c|}
\hline
\G & B_n & C_n & F_4 & G_2\\
\hline
\wh\G &A_{2n-1} & D_{n+1} & E_6 & D_4\\
\hline  
A(\G) & \ZM/2 & \ZM/2 & \ZM/2 & \SG_3\\
\hline
\end{array}
\]

In general, there is a unique (in case $\G = C_3$
or $G_2$ : up to conjugation by $\Aut(\wh\G) = \SG_3$) faithful action
of $A(\G)$ on $\wh\G$. One can see $\G$ as the quotient of $\wh\G$ by
$A(\G)$.

In all cases but $\G = C_3$, the group $A(\G)$ is the full group of
automorphisms of $\wh\G$. Note that $D_4$ is associated to $C_3$ and
$G_2$. For a homogeneous diagram, it will be convenient to set $\wh\G
= \G$ and $A(\G) = 1$.

A rational double point may be represented as the quotient $\AM^2/H$
of $\AM^2$ by a finite subgroup $H$ of $SL_2$ provided the
characteristic of $k$ is good for the corresponding Dynkin diagram. If
$\wh H$ is another finite subgroup of $SL_2$ containing $H$ as a
normal subgroup, then the quotient $\wh H/H$ acts naturally on
$\AM^2/H$.

\begin{definition}
Let $\G$ be an inhomogeneous irreducible Dynkin diagram and let the
characteristic of $k$ be good for $\G$. A couple $(X,A)$ consisting of
a normal surface singularity $X$ and a group $A$ of automorphisms of
$X$ is called a \emph{simple singularity of type $\G$} if it is
isomorphic (after Henselization) to a couple $(\AM^2/H, \wh H/H)$
according to the following table.

\[
\begin{array}{|c|c|c|c|c|}
\hline
\G & B_n & C_n & F_4 & G_2\\
\hline
H & \CG_{2n} & \DG_{4(n - 1)} & \TG & \DG_8\\
\hline  
\wh H & \DG_{4n} & \DG_{8(n - 1)} & \OG & \OG\\
\hline
\end{array}
\]
\end{definition}

Then $X$ is a rational double point of type $\wh\G$ and $A$ is
isomorphic to $A(\G)$. The action of $A$ on $X$ lifts in a unique way
to an action of $A$ on the resolution $\wt X$ of $X$. As $A$ fixes
the singular point of $X$, the exceptional divisor in $\wt X$ will be
stable under $A$. In this way, we recover the action of $A$ on $\wh \G$.
The simple singularities of inhomogeneous type can be characterized in
the following way.

\begin{proposition}
Let $\G$ be a Dynkin diagram of type $B_n$, $C_n$, $F_4$ or $G_2$, and
let the characteristic of $k$ be good for $\G$. Let $X$ be a rational
double point of type $\wh\G$ endowed with an action of $A(\G)$, free
on the complement of the singular point, and such that the induced
action on the dual diagram of the minimal resolution of $X$ coincides
with the associated action of $A(\G)$ on $\wh\G$. Then $(X,A)$ is a
simple singularity of type $\G$.
\end{proposition}

\subsection{Perverse extensions and decomposition numbers}
\label{subsec:dec simple}

Let $\G$ be any irreducible Dynkin diagram, and suppose the
characteristic of $k$ is good for $\G$. Let $\wh\G$ be the
associated homogeneous Dynkin diagram, $A(\G)$ the associated symmetry
group, and $H \subset \wh H$ the corresponding
finite subgroups of $SL_2$. We recall that, if $\G$ is already
homogeneous, then we take $\wh \G = \G$, $A(\G) = 1$ and $\wh H = H$.
We stratify the simple singularity $X =
\AM^2/H$ into two strata: the origin $\{0\}$ (the singular point), and
its complement $U$, which is smooth since $H$ acts freely on
$\AM^2\setminus\{0\}$.  We want to determine the stalks of the three
perverse extensions of the (shifted) constant sheaf $\EM$ on $U$, for
$\EM$ in $(\KM,\OM,\FM)$, and for the two perversities $p$ and $p_+$
in the case $\EM = \OM$. By the results of Section
\ref{sec:perverse}, this will allow us to determine a
decomposition number.

By the quasi-homogeneous structure of the equation defining $X$ in
$\AM^3$, we have a $\GM_m$-action on $X$ contracting $X$ to the
origin. We are in the situation of Proposition
\ref{prop:gencone}. Thus it is enough to calculate the
cohomology of $U$ with $\OM$ coefficients. The cases $\EM = \KM$ or
$\FM$ will follow.

Let $\wh\Phi$ be the root system corresponding to $\wh\G$,  in a real vector
space $\wh V$ of dimension equal to the rank $n$ of $\wh\G$. We identify the
set $\wh\D$ of vertices of $\wh\G$ with a basis of $\wh\Phi$. We denote by
$P(\wh\Phi)$ and $Q(\wh\Phi)$ the weight lattice and the root lattice of
$\wh V$. The finite abelian group $P(\wh\Phi)/Q(\wh\Phi)$ is the fundamental group of
the corresponding adjoint group, and also the center of the
corresponding simply-connected group. Its order is called the connection
index of $\wh\Phi$. The coweight lattice $P^\vee(\wh\Phi)$ (the weight lattice
of the dual root system $\wh\Phi^\vee$ in $\wh V^*$) is in duality with
$Q(\wh\Phi)$, and the coroot lattice $Q^\vee(\wh\Phi)$ is in duality with
$P(\wh\Phi)$. Thus the finite abelian group $P^\vee(\wh\Phi)/Q^\vee(\wh\Phi)$ is
dual to $P(\wh\Phi)/Q(\wh\Phi)$.

Let $\pi : \wt X \to X$ be the minimal resolution of $X$.
The exceptional divisor $E$ is the union of projective lines $E_\a$,
$\a\in\wh\D$. Then we have an isomorphism
$H^2(\wt X, \OM) \isom \OM \otimes_\ZM P(\wh\Phi)$
such that, for each $\a$ in $\wh\D$, the cohomology class of the
subvariety $E_\a$ is identified with $1 \otimes \a$, and such that the
intersection pairing is the opposite of the pullback of the $W$-invariant pairing on
$P$ normalized by the condition $(\a,\a) = 2$ for $\a$ in $\wh\D$ \cite{ItoNa}. 
Thus the natural map $H^2_c(\wt X, \OM) \to H^2_c(E,\OM)$ is
identified with the opposite of the map
$\OM \otimes_\ZM Q^\vee(\wh\Phi) \to \OM \otimes_\ZM P^\vee(\wh\Phi)$
induced by the inclusion.

By Poincaré duality ($U$ is smooth), it is enough to compute the
cohomology with proper support of $U$, and to do this we will use the
long exact sequence in cohomology with proper support for the open
subvariety $U$ with closed complement $E$ in $\wt X$.
The following table gives the $H^i_c(-,\OM)$ of the three varieties
(the first column is deduced from the other two).

\[
\begin{array}{c|c|c|c}
i & U & \wt X & E\\
\hline
0 & 0 & 0 & \OM\\
1 & \OM & 0 & 0\\
2 & 0 &
    \OM \otimes_\ZM Q^\vee(\wh\Phi) & \OM \otimes_\ZM P^\vee(\wh\Phi)\\
3 & \OM \otimes_\ZM P^\vee(\wh\Phi)/Q^\vee(\wh\Phi) & 0 & 0\\
4 & \OM & \OM & 0
\end{array}
\]


By (derived) Poincaré duality, we obtain the cohomology of $U$.

\begin{proposition}
The cohomology of $U$ is given by
\begin{equation}
\rg(U,\OM) \simeq \OM \oplus \OM \otimes_\ZM P(\wh\Phi)/Q(\wh\Phi) [-2] \oplus \OM [-3]
\end{equation}
\end{proposition}

The closed stratum is a point, and for complexes on the point the
perverse $t$-structures for $p$ and $p_+$ are the usual ones (there is no
shift since the point is $0$-dimensional).
With the notations of Subsection \ref{subsec:cones}, we have
\begin{gather}
\label{simple H -1}
H^{-1} i^* j_*(\OM[2]) \simeq H^1(U,\OM) = 0\\
\label{simple H 0}
H^0 i^* j_*(\OM[2]) \simeq H^2(U,\OM) \simeq \OM \otimes_\ZM P(\wh\Phi)/Q(\wh\Phi)\\
\label{simple H 1}
H^1 i^* j_*(\OM[2]) \simeq H^3(U,\OM) \simeq \OM
\end{gather}

By our analysis in Subsections \ref{subsec:tt recollement} and
\ref{subsec:F recollement}, we obtain the following results.

\begin{proposition}
We keep the preceding notation. In particular, $X$ is a simple
singularity of type $\G$.

Over $\KM$, we have canonical isomorphisms
\begin{equation}
\p j_! (\KM[2]) \simeq \p j_{!*} (\KM[2]) \simeq \p j_* (\KM[2])
\simeq \KM_X[2]
\end{equation}
In particular, $X$ is $\KM$-smooth.

Over $\OM$, we have canonical isomorphisms
\begin{gather}
\label{simple iso IC}
\p j_! (\OM[2])
\simeq \pp j_! (\OM[2])
\simeq \p j_{!*} (\OM[2])
\simeq \OM_X[2]
\\
\label{simple iso IC+}
\pp j_{!*} (\OM[2])
\simeq \p j_* (\OM[2])
\simeq \pp j_* (\OM[2])
\end{gather}
and a short exact sequence in $\p\MC(X,\OM)$
\begin{equation}
\label{simple ses}
0 \longto \p j_{!*} (\OM[2])
\longto \pp j_{!*} (\OM[2])
\longto i_* \OM \otimes_\ZM (P(\wh\Phi)/Q(\wh\Phi))
\longto 0
\end{equation}

Over $\FM$, we have canonical isomorphisms
\begin{gather}
\label{simple iso F IC}
\F\p j_!\ (\OM[2])
\isom \p j_!\ (\FM[2])
\isom \F\pp j_!\ (\OM[2])
\isom \F\p j_{!*}\ (\OM[2])
\isom \FM_X[2]
\\
\label{simple iso F IC+}
\F\pp j_{!*}\ (\OM[2])
\isom \F\p j_*\ (\OM[2])
\isom \p j_*\ (\FM[2])
\isom \F\pp j_*\ (\OM[2])
\end{gather}
and short exact sequences
\begin{gather}
\label{simple ses F IC}
0 \longto i_* \FM \otimes_\ZM (P(\wh\Phi)/Q(\wh\Phi))
\longto \F\p j_{!*} (\OM[2])
\longto \p j_{!*} (\FM[2])
\longto 0
\\
\label{simple ses F IC+}
0 \longto \p j_{!*}(\FM[2])
\longto \F\pp j_{!*}(\OM[2])
\longto i_* \FM \otimes_\ZM \left(P(\wh\Phi)/Q(\wh\Phi)\right)
\longto 0
\end{gather}

We have
\begin{equation*}
\label{simple dec}
[\F\p {j}_{!*}\, (\OM[2]) : i_* \FM]
= [\F\pp {j}_{!*}\, (\OM[2]) : i_* \FM]
= \dim_\FM \FM \otimes_\ZM \left(P(\wh\Phi)/Q(\wh\Phi)\right)
\end{equation*}

In particular, $\F\p j_{!*}\ (\OM[2])$ is simple
\textup{(}and equal to $\F\pp j_{!*}\ (\OM[2])$\textup{)}
if and only if $\ell$ does not divide the
connection index $|P(\wh\Phi)/Q(\wh\Phi)|$ of $\wh\Phi$. The variety
$X$ is $\FM$-smooth under the same condition.
\end{proposition}

\begin{proof}
Taking into account \eqref{simple H -1}, \eqref{simple H 0} and
\eqref{simple H 1}, the statements over $\KM$ follow from
the triangles \eqref{tri:! !*} and \eqref{tri:!* *},
the statements over $\OM$ follow from the triangles
\eqref{tri:rec tt first} to \eqref{tri:rec tt last},
and the statements over $\FM$ follow from the triangles
\eqref{tri:mod rec first} to \eqref{tri:mod rec last}.
The determination of the decomposition number follows.
\end{proof}

Let us give this decomposition number in each type:
\[
\begin{array}{l|l|l}
\wh\G & P(\wh\Phi)/Q(\wh\Phi) &
[\F \p {j}_{!*}\ (\OM[2]) : i_*\ \FM]\\
\hline
A_n & \ZM / (n + 1) &
      1 \text{ if } \ell \mid n + 1,\ 0 \text{ otherwise}\\
D_n\ (n \text{ even}) & (\ZM/2)^2 &
      2 \text{ if } \ell = 2,\ 0\text{ otherwise}\\
D_n\ (n \text{ odd}) & \ZM/4 &
      1 \text{ if } \ell = 2,\ 0\text{ otherwise}\\
E_6 & \ZM/3 &
      1 \text{ if } \ell = 3,\ 0\text{ otherwise}\\
E_7 & \ZM/2 &
      1 \text{ if } \ell = 2,\ 0\text{ otherwise}\\
E_8 & 0 & 0
\end{array}
\]

Let us note that for $\G = E_8$, the variety $X$ is $\FM$-smooth for
any $\ell$. However, it is not smooth, since it has a double point.

In the preceding calculations, the closed stratum was just a point,
and local systems on a point can be considered as $\EM$-modules.
However, for the next application (to the subregular orbit), 
non-trivial local systems may occur. For that reason, we have to
keep track of the action of $A(\G)$.

Let us first recall some facts
from \cite{BOUR456}. Let $\Aut(\wh\Phi)$ denote the group of
automorphisms of $\wh V$ stabilizing $\wh\Phi$. The subgroup of
$\Aut(\wh\Phi)$ of the elements stabilizing $\wh\D$ is identified with
$\Aut(\wh\G)$. The Weyl group $W(\wh\Phi)$ is a normal subgroup of
$\Aut(\wh\Phi)$, and $\Aut(\wh\Phi)$ is the semi-direct product of
$\Aut(\wh\G)$ and $W(\wh\Phi)$ \cite[Chap. VI, \S 1.5, Prop. 16]{BOUR456}.

The group $\Aut(\wh\Phi)$ stabilizes $P(\wh\Phi)$ and $Q(\wh\Phi)$,
thus it acts on the quotient $P(\wh\Phi)/Q(\wh\Phi)$. By
\cite[Chap. VI, \S 1.10, Prop. 27]{BOUR456}, the group $W(\wh\Phi)$
acts trivially on $P(\wh\Phi)/Q(\wh\Phi)$. Thus, the quotient group
$\Aut(\wh\Phi)/W(\wh\Phi) \simeq \Aut(\wh\G)$ acts canonically on
$P(\wh\Phi)/Q(\wh\Phi)$.

Now $A(\G)$ acts on $X$, $\wt X$, $E$ and $U$, and hence on their
cohomology (with or without supports). Moreover, the action of $A(\G)$
on $H^2_c(E,\OM) \simeq \OM \otimes_\ZM P^\vee(\wh\Phi)$ is the one
induced by the inclusions
$A(\G) \subset \Aut(\wt\G) \subset \Aut(\wh\Phi)$.
The inclusions of $E$ and $U$ in $\wt X$ are $A(\G)$-equivariant,
hence the maps in the long exact sequence in cohomology with compact
support that we considered earlier (to calculate $H^3_c(U,\OM)$)
are $A(\G)$-equivariant. Thus the action of $A(\G)$ on
$H^3_c(U,\OM) \simeq P^\vee(\wh\Phi)/Q^\vee(\wh\Phi)$ is induced by the
inclusion $A(\G) \subset \Aut(\G) \simeq \Aut(\wh\Phi)/W(\wh\Phi)$
from the canonical action. It follows that the action of $A(\G)$ on
$H^2(U,\G) \simeq P(\wh\Phi)/Q(\wh\Phi)$ also comes from the canonical
action of $\Aut(\wh\Phi)/W(\wh\Phi)$.

\subsection{Subregular class}
\label{subsec:subreg}

Let $G$ be a simple and adjoint algebraic group over $k$ of type $\G$.
We will recall some facts about the geometry of the subregular orbit
from \cite{SLO2}.
We assume that the characteristic of $k$ is $0$ or greater than $4h - 2$
(where $h$ is the Coxeter number). This is a serious restriction on
$p$, but it does not matter so much for our purposes. Note that, on
the other hand, we make no assumption on $\ell$ (the only restriction
is $\ell \neq p$).

Let $\NC$ denote the nilpotent cone in the Lie algebra $\gG$ of $G$.
Let $\OC_\reg$ (resp. $\OC_\subreg$) be the regular
(resp. subregular) orbit in $\NC$.
The  orbit $\OC_\subreg$ is the unique
open dense orbit in $\NC \setminus \OC_\reg$
(we assume that $\gG$ is simple). It is of
codimension 2 in $\NC$. Let $x_\reg \in \OC_\reg$ and
$x_\subreg \in \OC_\subreg$.

The centralizer of $x_\reg$ in $G$ is a connected unipotent subgroup,
hence $A_G(x_\reg) = 1$. The unipotent radical of the centralizer in
$G$ of $x_\subreg$ has a reductive complement $C$ given by the following
table.

\[
\begin{array}{c|c|c|c|c|c|c|c|c|c}
\G & A_n\ (n > 1) & B_n & C_n & D_n & E_6 & E_7 & E_8 & F_4 & G_2\\
\hline
C(x) & \GM_m & \GM_m \rtimes \ZM/2 & \ZM/2 & 1 & 1 & 1 & 1 & \ZM/2 & \SG_3
\end{array}
\]

In type $A_1$, the subregular class is just the trivial class, so in
this case the centralizer is $G = PSL_2$ itself, which is reductive.

We have $A_G(x_\subreg) \simeq C/C^0$. This group is isomorphic to
the associated symmetry group $A(\G)$ introduced in Subsection
\ref{subsec:symmetries}.

Let $X$ be the intersection $X = S \cap \NC$ of a transverse slice $S$ to the
orbit $\OC_\subreg$ of $x_\subreg$ with the nilpotent variety $\NC$.
The group $C$ acts on $X$. We can find a section $A$ of
$C/C^0 \simeq A_G \simeq A(\G)$ in $C$. In homogeneous types, $A$ is
trivial. If $\G = C_n$, $F_4$ or $G_2$, then $A = C$. If $\G = B_n$,
take $\{1,s\}$ where $s$ is a nontrivial involution (in this case, $A$
is well-defined up to conjugation by $C^0 = \GM_m$).

\begin{theorem}
\cite{BRI,SLO1,SLO2}
\label{th:subreg}
We keep the preceding notation.
The surface $X$ has a rational double point of type $\wh\G$ at $x_\subreg$.
Thus
\[
\Sing(\ov \OC_\reg, \OC_\subreg) = \wh\G
\]

Moreover the couple $(X,A)$ is a simple singularity of type $\G$.
\end{theorem}

In fact, the first part of the theorem is already true when the
characteristic of $k$ is very good for $G$. This part is enough to
calculate the decomposition numbers $d_{(x_\reg,1),(x_\subreg,1)}$ for
homogeneous types (then $A = 1$), and even some more decomposition
numbers $d_{(x_\reg,1),(x_\subreg,\rho)}$ for the other
types. Actually, what can be deduced in all types is the following
relation:
\[
\sum_{\rho\in\Irr\FM A} \rho(1) \cdot d_{(x_\reg,1),(x_\subreg,\rho)}
= \dim_\FM \FM \otimes_\ZM P(\wh\Phi)/Q(\wh\Phi)
\]
This is enough, for example, to determine for which $\ell$ we have
\[
\forall \rho \in \Irr \FM A,\ 
d_{(x_\reg,1),(x_\subreg,\rho)} = 0
\]
(those $\ell$ are the ones
which do not divide the connection index of $\wh\Phi$).

Anyway, the second part of the theorem will allow us to deal with the
local systems involved on $\OC_\subreg$.

Let $\jreg : \OC_\reg \injto \OC_\reg \cup \OC_\subreg$ be the open immersion,
and $i_\subreg : \OC_\subreg \injto \OC_\reg \cup \OC_\subreg$ the closed
complement. Finally, let $j$ be the open inclusion of $\OC_\subreg
\cup \OC_\reg$ into $\NC$. Applying the functor $j^*$, we see that
\begin{eqnarray*}
d^\NC_{(x_\reg,1),(x_\subreg,\rho)}
&:=& [\FM \p \JC_{!*} (\OC_\reg, \OM) :
   \p \JC_{!*} (\OC_\subreg, \rho)]\\
&=& [\FM \jreg_{!*} (\OM[2\nu]) : \isubreg_{*} (\rho [2\nu + 2])]
\end{eqnarray*}

By Slodowy's theorem and the analysis of Subsection \ref{subsec:dec simple}, we
obtain the following result:

\begin{theorem}
We have
\begin{equation*}
d^\NC_{(x_\reg,1),(x_\subreg,\rho)} 
= [\FM \otimes_\ZM P(\wh\Phi)/Q(\wh\Phi) : \rho]
\end{equation*}
for all $\rho$ in $\Irr \FM A$.
\end{theorem}

For homogeneous types, we recover the decomposition numbers described
in Subsection \ref{subsec:dec simple}. Let us describe in detail all the other
possibilities. The action of $\Aut(\wh\Phi)/W(\wh\Phi)$ on
$P(\wh\Phi)/Q(\wh\Phi)$ is described in all types in
\cite[Chap. VI, \S 4]{BOUR456}.

In the types $B_n$, $C_n$ and $F_4$, we have $A \simeq
\ZM/2$. When $\ell = 2$, we have $\Irr \FM A = \{1\}$. In this case,
we would not even need to know the actual action, since for our
purposes we only need the class in the Grothendieck group
$K_0(\FM A) \simeq \ZM$, that is, the dimension. When $\ell$ is not
$2$, we have $\Irr \FM A = \{1,\e\}$, where $\e$ is the unique
non-trivial character of $\ZM/2$.

\subsubsection{Case $\G = B_n$}
We have $\wh\G = A_{2n - 1}$ and
$P(\wh\Phi)/Q(\wh\Phi) \simeq \ZM/2n$. The non-trivial element of
$A \simeq \ZM/2$ acts by $-1$. Thus we have

\[
\begin{array}{lllll}
\text{If } \ell = 2, &\text{then}& d^\NC_{(x_\reg,1),(x_\subreg,1)} = 1\\
\text{If } 2 \neq \ell \mid n, &\text{then}& d^\NC_{(x_\reg,1),(x_\subreg,1)} = 0
&\text{and}& d^\NC_{(x_\reg,1),(x_\subreg,\e)} = 1\\
\text{If } 2 \neq \ell \nmid n, &\text{then}& d^\NC_{(x_\reg,1),(x_\subreg,1)} = 0
&\text{and}& d^\NC_{(x_\reg,1),(x_\subreg,\e)} = 0
\end{array}
\]

\subsubsection{Case $\G = C_n$}

We have $\wh\G = D_{n + 1}$.

If $n$ is even, then we have
$P(\wh\Phi)/Q(\wh\Phi) \simeq \ZM/4$,
and the nontrivial element of $A \simeq \ZM/2$ acts by $-1$.

If $n$ is odd, then we have
$P(\wh\Phi)/Q(\wh\Phi) \simeq (\ZM/2)^2$,
and the nontrivial element of $A \simeq \ZM/2$ acts by exchanging two
nonzero elements.

Thus we have
\[
\begin{array}{lll}
\text{If } \ell = 2 \text{ and $n$ is even},& \text{then}& d^\NC_{(x_\reg,1),(x_\subreg,1)} = 1\\
\text{If } \ell = 2 \text{ and $n$ is odd},& \text{then}& d^\NC_{(x_\reg,1),(x_\subreg,1)} = 2\\
\text{If } \ell \neq 2, &\text{then}& 
d^\NC_{(x_\reg,1),(x_\subreg,1)} = d^\NC_{(x_\reg,1),(x_\subreg,\e)} = 0
\end{array}
\]

\subsubsection{Case $\G = F_4$}

We have $\wh\G = E_6$ and $P(\wh\Phi)/Q(\wh\Phi) \simeq \ZM/3$.
The nontrivial element of $A \simeq \ZM/2$ acts by $-1$.
Thus we have
\[
\begin{array}{lllll}
\text{If } \ell = 2, &\text{then}&  d^\NC_{(x_\reg,1),(x_\subreg,1)} = 0\\
\text{If } \ell = 3, &\text{then}&  d^\NC_{(x_\reg,1),(x_\subreg,1)} = 0
&\text{and}& d^\NC_{(x_\reg,1),(x_\subreg,\e)} = 1\\
\text{If } \ell > 3, &\text{then}& d^\NC_{(x_\reg,1),(x_\subreg,1)} = 0
&\text{and}& d^\NC_{(x_\reg,1),(x_\subreg,\e)} = 0
\end{array}
\]

\subsubsection{Case $\G = G_2$}

We have $\wh\G = D_4$ and $P(\wh\Phi)/Q(\wh\Phi) \simeq (\ZM/2)^2$.
The group $A \simeq \SG_3$ acts by permuting the three non-zero
elements. Let us denote the sign character by $\e$ (it is nontrivial
when $\ell \neq 2$), and the degree two character by $\psi$ (it
remains irreducible for $\ell = 2$, but for $\ell = 3$ it decomposes
as $1 + \e$). We have
\[
\begin{array}{lll}
\text{If } \ell = 2, &\text{then}&  d^\NC_{(x_\reg,1),(x_\subreg,1)} = 0
\quad\text{and}\quad d^\NC_{(x_\reg,1),(x_\subreg,\psi)} = 1\\
\text{If } \ell = 3, &\text{then}&  d^\NC_{(x_\reg,1),(x_\subreg,1)} = 
 d^\NC_{(x_\reg,1),(x_\subreg,\e)} = 0\\
\text{If } \ell > 3, &\text{then}& d^\NC_{(x_\reg,1),(x_\subreg,1)} = 
d^\NC_{(x_\reg,1),(x_\subreg,\e)} = d^\NC_{(x_\reg,1),(x_\subreg,\psi)} = 0
\end{array}
\]

\section{Minimal singularities}
\label{sec:min}

Let $G$ be as in the last section. We assume that $p$ is good. We consider the unique
(non-trivial) minimal nilpotent orbit $\OC_\mini$ in $\gG$ (it is the
orbit of a highest weight vector for the adjoint representation).
It is of dimension $d = 2 h^\vee - 2$, where $h^\vee$
is the dual Coxeter number \cite{WANG}.

Its closure $\ov \OC_\mini = \OC_\mini \cup \{0\}$ is a cone with
origin $0$.  Let $\jmin : \OC_\mini \to \ov \OC_\mini$ be the open
immersion, and $\io : \{0\} \to \ov \OC_\mini$ the closed
complement. By Proposition \ref{prop:cone}, we have
\[
\io^* \jmin_* (\OM[d]) \simeq \bigoplus_i H^{i + d}(\OC_\mini,\OM) [-i]
\]

Let $\Phi$ denote the root system of $\gG$ and let us choose some
basis of $\Phi$. Let $\Phi'$ be the root subsystem of $\Phi$ generated
by the long simple roots.  In \cite{cohmin}, we computed the
cohomology of $\OC_\mini$ over $\OM$. In particular, we
obtained the following results:
\begin{gather}
\label{min H -1}
H^{-1} \io^* \jmin_* (\OM[d]) = H^{d-1} (\OC_\mini, \OM) = 0
\\
\label{min H 0}
H^0 \io^* \jmin_* (\OM[d])
 = H^{d} (\OC_\mini, \OM)
 = \OM \otimes_\ZM (P^\vee(\Phi')/Q^\vee(\Phi'))
\\
\label{min H 1}
H^1 \io^* \jmin_* (\OM[d]) = H^{d+1} (\OC_\mini, \OM)
\text{ is torsion-free}
\end{gather}

By the distinguished triangles in Subsections
\ref{subsec:tt recollement} and \ref{subsec:F recollement},
we obtain the following:
\begin{theorem}
Over $\OM$, we have canonical isomorphisms
\begin{gather*}
\label{min iso IC}
\p \jmin_! (\OM[d])
\simeq \pp \jmin_! (\OM[d])
\simeq \p \jmin_{!*} (\OM[d])
\\
\label{min iso IC+}
\pp \jmin_{!*} (\OM[d])
\simeq \p \jmin_* (\OM[d])
\simeq \pp \jmin_* (\OM[d])
\end{gather*}
and a short exact sequence
{\small
\begin{equation*}
\label{min ses}
0 \longto \p \jmin_{!*} (\OM[d])
\longto \pp \jmin_{!*} (\OM[d])
\longto \io_* \OM \otimes_\ZM (P^\vee(\Phi')/Q^\vee(\Phi'))
\longto 0
\end{equation*}
}

Over $\FM$, we have canonical isomorphisms
{\small
\begin{gather*}
\label{min iso F IC}
\F \p \jmin_!(\OM[d])
\isom \p \jmin_!(\FM[d])
\isom \F\pp \jmin_!(\OM[d])
\isom \F\p \jmin_{!*}(\OM[d])
\\
\label{min iso F IC+}
\F\pp \jmin_{!*}(\OM[d])
\isom \F\p \jmin_*(\OM[d])
\isom \p \jmin_*(\FM[d])
\isom \F\pp \jmin_*(\OM[d])
\end{gather*}
}
and short exact sequences
{\small
\begin{gather*}
\label{min ses F IC}
0 \longto \io_*\FM \otimes_\ZM (P^\vee(\Phi')/Q^\vee(\Phi'))
\longto \F\p \jmin_{!*}(\OM[d])
\longto \p \jmin_{!*}(\FM[d])
\longto 0
\\
\label{min ses F IC+}
0 \longto \p \jmin_{!*}(\FM[d])
\longto \F\pp \jmin_{!*}(\OM[d])
\longto \io_*\FM \otimes_\ZM (P^\vee(\Phi')/Q^\vee(\Phi'))
\longto 0
\end{gather*}
}

We have
{\scriptsize
\begin{equation*}\label{min dec}
[\F\p {j_\mini}_{!*}(\OM[d]) : \io_*\FM]
= [\F\pp {j_\mini}_{!*}(\OM[d]) : \io_*\FM]
= \dim_\FM \FM \otimes_\ZM \left( P^\vee(\Phi')/Q^\vee(\Phi') \right)
\end{equation*}
}

In particular, $\F\p {j_\mini}_{!*}(\OM[d])$ is simple 
(and equal to $\F\pp {j_\mini}_{!*}(\OM[d])$) if and only if
$\ell$ does not divide the connection index of $\Phi'$.
\end{theorem}

Let us give this decomposition number in each type.
We denote the singularity of $\ov\OC_\mini$ at the origin by the lower
case letter corresponding to the type $\G$ of $\gG$.

\[
\begin{array}{l|l|l|l}
\textrm{Singularity} & \G' & P^\vee(\Phi')/Q^\vee(\Phi') & d^\NC_{(x_\mini,1),(0,1)}\\
\hline
a_n & A_n & \ZM / (n + 1) &
      1 \text{ if } \ell \mid n + 1,\ 0 \text{ otherwise}\\
b_n & A_{n - 1} & \ZM/n &
      1 \text{ if } \ell \mid n,\ 0 \text{ otherwise}\\
c_n & A_1 & \ZM/2 &
      1 \text{ if } \ell = 2,\ 0 \text{ otherwise}\\
d_n\ (n \text{ even}) & D_n & (\ZM/2)^2 &
      2 \text{ if } \ell = 2,\ 0\text{ otherwise}\\
d_n\ (n \text{ odd}) & D_n & \ZM/4 &
      1 \text{ if } \ell = 2,\ 0\text{ otherwise}\\
e_6 & E_6 & \ZM/3 &
      1 \text{ if } \ell = 3,\ 0\text{ otherwise}\\
e_7 & E_7 & \ZM/2 &
      1 \text{ if } \ell = 2,\ 0\text{ otherwise}\\
e_8 & E_8 & 0 & 0\\
f_4 & A_2 & \ZM/3 &
      1 \text{ if } \ell = 3,\ 0 \text{ otherwise}\\
g_2 & A_1 & \ZM/2 &
      1 \text{ if } \ell = 2,\ 0 \text{ otherwise}
\end{array}
\]

Here we only used the $H^i(\OC_\mini,\OM)$ for $i = d-1$, $d$, $d + 1$,
but in \cite{cohmin}, we computed all of the cohomology of $\OC_\mini$, so
if the reader is interested, one can deduce from that all the stalks
of the perverse extensions. In particular, there is 
torsion in the stalks of $\p \jmin_! (\OM [d])$ only if $\ell$ is bad for $G$.
Note that the singularities $c_n$ (for $n \geqslant 1$, including
$c_1 = a_1 = A_1$ and $c_2 = b_2$) and $g_2$ are $\KM$-smooth but not
$\FM_2$-smooth (actually the latter is not $\FM_3$-smooth either).

\def\cprime{$'$}

\end{document}